\documentclass{amsart}

\usepackage[utf8]{inputenc}
\usepackage{hyperref}
\hypersetup{colorlinks=true,
	linkcolor=black, 
	urlcolor=blue} 
\usepackage{amsmath,amssymb,amsfonts,amsthm}
\usepackage{amsthm}
\usepackage{enumerate}
\numberwithin{equation}{section}
\usepackage{marginnote}
\usepackage[a4paper]{geometry}
\usepackage{pdfrender,xcolor}

\reversemarginpar

\usepackage[english]{babel}
\usepackage[babel, english=american]{csquotes}
\newtheoremstyle{Teorema}
{3pt}
{3pt}
{\slshape}
{}
{\bfseries}
{:}
{\newline}
{}

\newtheorem{theorem}{Theorem}[section]
\newtheorem{definition}{Definition}

\newtheorem{corollary}[theorem]{Corollary}

\newtheorem{lemma}[theorem]{Lemma}

\theoremstyle{definition}

\DeclareMathOperator{\D}{d}
\DeclareMathOperator{\supp}{supp}
\DeclareMathOperator{\Co}{Co}

\newcommand{\Z}{\mathbb{Z}}
\newcommand{\R}{\mathbb{R}}
\newcommand{\C}{\mathbb{C}}

\newcommand{\n}{\vec n}

\title{On the convergence  of   multi-level Hermite-Pad\'e approximants }
\author{L.G. González Ricardo, \,\,\,G. L\'opez Lagomasino, \\ and S. Medina Peralta}

\begin{document}

    \pdfrender{StrokeColor=black,TextRenderingMode=2,LineWidth=0.4pt}

\begin{abstract}

	In the present paper we prove a Stieltjes type theorem on the convergence of a sequence of rational functions associated with a mixed type Hermite-Pad\'e approximation problem of a Nikishin system of functions and analyze the ratio asymptotic of the corresponding Hermite-Pad\'e polynomials.

\end{abstract}	

	\maketitle


\section{Introduction}

    In \cite{Lund}, a mixed type Hermite-Pad\'e approximation problem was introduced in order to find    multipeakon
     solutions of the Degasperis-Procesi equation. The same construction is relevant in the study of the two matrix model \cite{BGS} and Cauchy  biorthogonal polynomials \cite{Bertola:CBOPs}.
     The functions approximated are the Cauchy transforms of a pair of measures forming a Nikishin system. This motivated in \cite{Lago_Sergio_Jacek} the extension of the approximation problem to the case of Nikishin systems with an arbitrary number of measures and study their convergence properties.   Previously, the authors of \cite {BB}  extended the two matrix model to an $m-$matrix model   and   studied   a Riemann-Hilbert problem that leads to  the mixed type Hermite-Pad\'e approximation problem introduced in \cite{Lago_Sergio_Jacek}.
      Recently, the definition   was further extended by V.G. Lysov in \cite{Lysov} and a number of interesting properties revealed. Before stating our goals, let us briefly present the so called Nikishin systems of measures and Lysov's definition of multi-level Hermite-Pad\'e approximants.

\subsection{Nikishin systems}
    Nikishin systems of measures were first introduced by E.M. Nikishin in \cite{nikishin}. Throughout the years, it has been shown that polynomials which share orthogonality relations with respect to such systems of measures behave asymptotically in a way similar to standard orthogonal polynomials (with respect to a single measure). A brief overview of the subject may be found in \cite{lago2}. 

    In the sequel, we will only consider Borel measures $s$ with constant sign, finite moments $c_n = \int x^n \D s (x) \neq \infty,  n\in\Z_+$, whose support $\supp s\subset \R$ consists of infinitely many points. We will denote by $\Delta$ the smallest interval which contains $\supp s$ ($\Co(\supp  s)=\Delta$). The class of these measures will be denoted by $\mathcal{M}(\Delta)$. Let
    $$\widehat{s}(z) = \int \frac{\D s(x)}{z-x}$$
    denote the Cauchy transform of the measure $s$. Obviously, $\widehat{s} \in \mathcal{H}(\overline{\C}\setminus\Delta)$ (holomorphic  in $\overline{\C}\setminus\Delta$). We can associate to $\widehat{s}$ its formal Laurent expansion at infinity
    $$\widehat{s}(z)\sim \sum_{j=0}^\infty\frac{c_j}{z^{j+1}},\qquad c_j = \int x^j\D s(x).$$

    Let $\Delta_\alpha$, $\Delta_\beta$ be two intervals contained in the real line such that $\Delta_\alpha\cap\Delta_\beta$ contains at most one common point. Assume that $\sigma_\alpha\in\mathcal{M}(\Delta_\alpha), \sigma_\beta\in\mathcal{M}(\Delta_\beta)$, and $\widehat{\sigma}_\beta \in L_1(\sigma_\alpha)$ (integrable with respect to $\sigma_\alpha$). With these two measures we construct a third one, denoted $\langle \sigma_\alpha,\sigma_\beta\rangle$, whose differential notation is given by
    $$\D\langle \sigma_\alpha,\sigma_\beta\rangle(x): = \widehat{\sigma}_\beta(x)\D\sigma_\alpha(x).$$

    When considering consecutive products of measures, for example $\langle \sigma_\alpha,\sigma_\beta, \sigma_\gamma \rangle := \langle\sigma_\alpha, \langle \sigma_\beta, \sigma_\gamma \rangle \rangle$, we implicitly assume not only that $\widehat{\sigma}_\gamma \in L_1(\sigma_\beta)$, but also $\langle \sigma_\beta, \sigma_\gamma\widehat{\rangle}\in L_1(\sigma_\alpha)$, where $\langle \sigma_\beta, \sigma_\gamma\widehat{\rangle}$ denotes the Cauchy transform of $\langle \sigma_\beta, \sigma_\gamma\rangle$. Notice that this product is neither commutative nor associative.

    \begin{definition}
        \label{Nikishin_sys}
        Take a collection $\Delta_j$, $j=1,\ldots,m,$ of intervals such that
        $\Delta_j\cap\Delta_{j+1}, $ $ j=1,\ldots,m-1$ contains at most one point.
        Let $(\sigma_1,\ldots,\sigma_m)$ be a system of measures such that $\Co(\supp \sigma_j)=\Delta_j$, $\sigma_j\in\mathcal{M}(\Delta_j)$, $j=1,\ldots,m$. We say that $(s_{1,1},\ldots,s_{1,m})=\mathcal{N}(\sigma_1,\ldots,\sigma_m)$, where
        $$s_{1,1}=\sigma_1,\quad s_{1,2}=\langle\sigma_1,\sigma_2 \rangle,\quad \ldots, \quad s_{1,m}=\langle \sigma_1, \langle \sigma_2,\ldots,\sigma_m \rangle\rangle,$$
        is the Nikishin system of measures generated by $(\sigma_1,\ldots,\sigma_m)$. The vector $(\widehat{s}_{1,1},\ldots,\widehat{s}_{1,m})$ is the corresponding Nikishin system of functions.
    \end{definition}

    Notice that any sub-system of $(\sigma_1,\ldots,\sigma_m)$ of consecutive measures is also a generator of a Nikishin system. In the sequel, for $1\leq j\leq k\leq m$ we write
    $$s_{j,k} := \langle \sigma_j, \sigma_{j+1},\ldots, \sigma_k \rangle,\qquad s_{k,j} := \langle \sigma_k, \sigma_{k-1},\ldots, \sigma_j \rangle.$$
    With the system of measures $(\sigma_1,\ldots,\sigma_m)$ we can also define the reversed Nikishin system $(s_{m,m}, \ldots, s_{m,1}) = \mathcal{N}(\sigma_m,\ldots,\sigma_1)$.

\subsection{Multilevel Hermite-Pad\'e polynomials}
    We use the definition given in \cite[Problem A]{Lysov}.

    Let $(\Z_+^m)^*$ be the set of all $m$-dimensional vectors with non-negative integer components, not identically equal to zero. For $\vec n = (n_1,\ldots,n_{m})\in (\Z_+^m)^*$ we define $|\vec n|= n_1 + \cdots + n_{m}$.
    \begin{definition}\label{HP}
        Consider the Nikishin system $\mathcal{N}(\sigma_1,\ldots,\sigma_m)$ and $\vec{n} = (n_1,\ldots,n_{m}) \in (\Z_+^m)^*$. There exist polynomials $a_{\vec{n},0}, a_{\vec{n},1},\ldots,a_{\vec{n},m}$,  where $\deg a_{\vec{n},j}\leq |\vec n|-1, j=0,1,\ldots,m-1,$ and $\deg a_{\vec n,m}\leq |\vec n|$, not all identically equal to zero, called multi-level (ML) Hermite-Pad\'e polynomials of $\mathcal{N}(\sigma_1,\ldots,\sigma_m)$ with respect to $\vec n$, that verify
        \begin{equation}
            \mathcal{A}_{\vec{n},j}(z) := \left((-1)^j a_{\vec n,j} + \sum_{k=j+1}^m (-1)^k a_{\vec n,k}\widehat{s}_{j+1,k}\right)(z) =  \mathcal{O}\left(\frac{1}{z^{n_{j+1}+1}}\right), \qquad z \to \infty,\label{Problem}
        \end{equation}
        where $j=0,\ldots,m-1$ (the asymptotic expansion of $\mathcal{A}_{\n,j}$ at $\infty$ begins with $z^{-n_{j+1} -1}$, or higher). For completeness, set $\mathcal{A}_{\vec{n},m} := (-1)^m a_{\vec n,m}$.
    \end{definition}

  We warn the reader that with our terminology in \cite[Problem A]{Lysov} the ML Hermite-Pad\'e polynomials were defined with respect to the system $\mathcal{N}(\sigma_m,\ldots,\sigma_1)$.

When $m =1$ the  definition reduces to that of classical Pad\'e approximation, which  plays a central role in the solution of the inverse spectral problem for a  discrete
string with Dirichlet boundary condition, see \cite{BSJ,stieltjes}. When $m=2$ and $\vec n = (n, 0)  $  definition \ref{HP}  reduces to the Hermite-Pad\'e approximation problem used in the solution of the inverse spectral problem for the   discrete cubic string, see \cite{Lund}.   For an arbitrary $m$ and $\vec n = (n, 0,\ldots ,0)$, one obtains the original definition of ML Hermite-Pad\'e polynomials given in \cite{Lago_Sergio_Jacek}.
We suspect that the  ML Hermite-Pad\'e approximations  studied here, or the  particular case considered in \cite{Lago_Sergio_Jacek}, may be connected with the solution of the inverse spectral problem for the  discrete $(m+1)-$string.

    In this scheme of approximation, the interpolation conditions involve all the Nikishin systems of the ``inner levels'', i.e. $\mathcal{N}(\sigma_1,\ldots,\sigma_m)$, $\mathcal{N}(\sigma_2,\ldots,\sigma_m)$, \ldots, $\mathcal{N}(\sigma_m) = (s_{m,m})$. Finding the polynomials $a_{\vec n,0}, a_{\vec n,1},\ldots,a_{\vec n,m}$ reduces to solving a homogeneous linear system of $|\vec n|(m+1)$ equations, given by the interpolation conditions, on $|\vec n|(m+1)+1$ unknowns, corresponding with the coefficients of the polynomials. Consequently, the system of equations has a non trivial solution.

    Following K. Mahler \cite{Mahler}, a multi-index $\vec n \in (\Z_+^m)^*$ is said to be normal if $\deg a_{{\vec n},j} = |\vec n| -1, j=0,\ldots,m-1,$ and $\deg a_{{\vec n},m} = |\vec n|$ (that is, when all the polynomials have maximum degree possible). The system of functions is said to be perfect when all the multi-indices are normal. In \cite[Theorem 1.1]{Lysov}, it was proved that the Nikishin system of functions is perfect for this approximation problem. Normality implies that the vector $(a_{\vec n,0},\ldots,a_{\vec n,m})$ is uniquely determined up to a multiplicative factor. In the sequel, we normalize this vector so that $a_{\vec n,m}$ has leading coefficient equal to one.

    A sequence of multi-indices $\Lambda \subset (\Z_+^m)^*$ is called a ray sequence when $\lim_{\vec n \in \Lambda} n_j/|\vec n|$ exists for all $j=1,\ldots,m$. When the $\Delta_j$ are bounded non-intersecting intervals, and $\sigma_j'  \neq 0$, a.e. in $\Delta_j$, $j=1,\ldots,m,$ in \cite[Theorem 1.2]{Lysov} the logarithmic asymptotic of ray sequences of ML polynomials was obtained. Using that result, it was also proved \cite[Proposition  1.2]{Lysov} that
    \begin{equation}
        \label{markov}
        \lim_{\vec n \in \Lambda} \frac{a_{\vec n,j}}{a_{\vec n,m}} = \widehat{s}_{m,j+1}, \qquad j=0,\ldots,m-1
    \end{equation}
    uniformly on each compact subset of $\overline{\C} \setminus \Delta_m$ (with geometric rate). This is a Markov type theorem, see \cite{markov}. Notice that the limits belong to the Nikishin system of functions corresponding to $\mathcal{N}(\sigma_m,\ldots,\sigma_1)$. It should be said that the proof of these results given in \cite{Lysov} may be adapted to the case when the measures $\sigma_j$ are in the much wider class of regular measures (for the definition of a regular measure see \cite[Chapter 3]{Stahl_Totik}).

    We  provide a convergence result such as \eqref{markov} in which the intervals $\Delta_j$ may be unbounded and consecutive intervals can have a common end point. This situation appears in \cite{BGS} in relation with the study of the two matrix model.

    When the $\Delta_j$ are bounded non-intersecting intervals, and $\sigma_j' \neq 0$ a.e. in $\Delta_j$, we also give a result about the asymptotic of sequences of ratios of polynomials $a_{\vec n,j}$ corresponding to ``consecutive'' multi-indices which resembles E.A. Rakhmanov's celebrated theorem on the ratio asymptotic of standard orthogonal polynomials (see \cite{Rak1, Rak2, Rak3, Nevai}). With the original definition of ML Hermite-Pad\'e polynomials introduced in  \cite{Lago_Sergio_Jacek} the ratio asymptotic was obtained in \cite{Lago_Sergio_Ulises}.

\subsection{Statement of the main results}
    A measure $s \in \mathcal{M}(\Delta)$ is said to satisfy Carleman's condition when the sequence of its moments $(c_n)_{n\geq 0}$ verifies
    \[ \sum_{n=0}^{\infty} |c_n|^{-1/2n} = \infty.
    \]
    When $\Delta$ is either $\R_+$ of $\R_-$, it is well known, see \cite{Car}, that this condition implies that the moment problem for $(c_n)_{n\geq 0}$ is determinate (that is, there is only one measure with that collection of moments). In turn, if the moment problem for $s$ is determinate then the sequence of diagonal Pad\'e approximants converges to $\widehat{s}$ on each compact subset of $\C \setminus \Delta$ (see T.J. Stieltjes \cite{stieltjes}). We prove the following Carleman-Stieltjes type theorem in the context of ML Hermite-Pad\'e approximants.
    \begin{theorem}
        \label{Convergence}
        Let $\Lambda \subset (\Z_+^{m})^*$ be an infinite sequence of distinct multi-indices for which there exist $\ell\in \{0,\ldots,m-2\}$ and a (fixed) non-negative integer $N$ such that $n_{j+1} \leq n_{j} + N$  for all $\ell+1 \leq j \leq m-1$ and $\vec n \in \Lambda$. Consider the sequence of vector polynomials $(a_{\vec n,0},\ldots,a_{\vec n,m})_{\vec n \in \Lambda}$ associated with  $\mathcal{N}(\sigma_1,\ldots,\sigma_m)$. For $j=\ell,\ldots,m-2$ the polynomial $a_{\n,j}$ has at least $  |\n|-2m- N\frac{m(m+1)}{2}$ sign changes in $\mathring{\Delta}_m$ (the interior of $\Delta_m$ with the Euclidean topology of $\R$). The polynomials $a_{\n,m-1}$ and $a_{\n,m}$ have, respectively, $|\n|-1$ and $|\n|$ interlacing simple zeros in $\mathring{\Delta}_m$.
        Suppose that either the sequence of moments of $\sigma_m$ satisfies Carleman's condition or $\Delta_{m-1} $ is  a bounded interval which does not intersect $\Delta_m$; then \eqref{markov} holds uniformly on each compact subset of $ \overline{\C} \setminus\Delta_m$ for $j = \ell,\ldots,m-1$. If $\Lambda \subset (\Z_+^{m})^*$ is an arbitrary infinite sequence of distinct multi-indices and  $\sigma_m$ satisfies Carleman's condition or, $\Delta_{m-1} $ is  a bounded interval which does not intersect $\Delta_m$ and $\lim_{\vec n \in \Lambda} (n_1 +\cdots + n_{m-1}) = \infty$, then \eqref{markov} takes place for $j= m-1$.
    \end{theorem}

If $\Lambda$ is a sequence of distinct multi-indices whose components are decreasing, the (first) condition on $\Lambda$ in Theorem \ref{Convergence} is verified with $\ell=0$ and $N=0$.  More precise information regarding the zeros of the polynomials $a_{\n,j}, j=0,\ldots,m-2$ will be given in Section 2.

    Let $\vec{n} \in \mathbb{Z}_{+}^{m}$ and $l \in\{1, \ldots, m\} .$ Define
    $$
    \vec{n}^{l}:=\left(n_{1}, \ldots, n_{l}+1, \ldots, n_{m}\right),
    $$
    the multi-index obtained adding $1$ to the $l$-th component of $\n$.

    \begin{theorem}
    	\label{th:ratiom}
        Consider the Nikishin system $\mathcal{N}(\sigma_1,\ldots,\sigma_m)$ where the  $\Delta_k,k=1,\ldots,m,$ are bounded, disjoint intervals, and $\sigma_k' \neq 0$ a.e. in $\Delta_k$. Let $\Lambda \subset (\Z_+^m)^*$ be an infinite sequence of distinct multi-indices for which  there exists a non-negative integer $N$ such that $n_{j+1} \leq n_{j} + N$  for all $1 \leq j \leq m-1$ and $\vec n \in \Lambda$.
        Then, for $k=0,\ldots, m$
	    \begin{equation}
            \label{left*}
	        \lim_{\n \in \Lambda} \frac{a_{\n^l,k}(z)}{a_{\n,k}(z)} = \frac{{\psi}^{(l)}_m(z)}{({\psi}^{(l)}_m)'(\infty)},
	    \end{equation}
	    uniformly on each compact subset of $\C\setminus\Delta_m$, where ${\psi}^{(l)}_m \in \mathcal{H}(\overline{\C} \setminus\Delta_m)$ is defined in \eqref{branches}.
    \end{theorem}

    The paper is organized as follows. In Section \ref{sec:conv} we prove Theorem \ref{Convergence}.  In Section \ref{sec:ratio} we study the asymptotic properties of the forms $\mathcal{A}_{\n,k}$ and the polynomials $a_{\n,k}$. In particular, Theorem \ref{th:ratiom} is derived from the more general Theorem \ref{th:ratio}. In the analysis of the ratio asymptotic an associated Riemann surface comes into play which will be introduced in that section.

\section{Convergence of the ML Hermite-Pad\'e approximants}\label{sec:conv}

\subsection{Preliminaries}
    First, we study the location of the zeros of the forms $\mathcal{A}_{\vec n,j}$, $j=0,\ldots,m$. For that purpose, we use \cite[Lemma 2.1]{Lago_Sergio} which we state here for convenience of the reader.

    \begin{lemma}
        \label{lm:rem}
        Let $(s_{1,1},\ldots,s_{1,m}) =\mathcal{N}(\sigma_1,\ldots,\sigma_m)$ be given. Assume that there exist polynomials with real coefficients $a_0,\ldots, a_m$ and a polynomial $w$ with real coefficients whose zeros lie in $\C\setminus\Delta_1$ such that
        \begin{equation*}
            \frac{\mathcal{A}(z)}{w(z)}\in\mathcal{H}(\C\setminus\Delta_1) \quad \textrm{ and }\quad \frac{\mathcal{A}(z)}{w(z)}=\mathcal{O}\left(\frac{1}{z^N}\right),\quad z\rightarrow \infty,
        \end{equation*}
        where $\displaystyle \mathcal{A}:= a_0 +\sum_{k=1}^m a_k\widehat{s}_{1,k}$ and $N\geq 1$. Let $\displaystyle \mathcal{A}_1:= a_1 + \sum_{k=2}^m a_k\widehat{s}_{2,k}$. Then,
        \begin{equation*}
            \frac{\mathcal{A}(z)}{w(z)} = \int \frac{\mathcal{A}_1(x)}{z-x}\frac{\D \sigma_1(x)}{w(x)}.
        \end{equation*}
        If $N\geq 2$, we also have
        \begin{equation}
            \label{A:orth}
            \int x^\nu \mathcal{A}_1(x)\frac{\D \sigma_1(x)}{w(x)}=0,\qquad \nu=0,1,\ldots, N-2.
        \end{equation}
        In particular, $\mathcal{A}_1$ has at least $N-1$ sign changes in $\mathring{\Delta}_1$ (the interior of $\Delta_1$ in $\R$ with the usual topology).
    \end{lemma}

    Given $\vec n = (n_1\ldots,n_{m}) \in \Z_+^m$, set
    \begin{align*}
	    \eta_{\vec{n},j} &:= n_1+\cdots+n_j.
    \end{align*}

    \begin{lemma}
        \label{lm:zeros}
        The form $\mathcal{A}_{\vec n,0}$ has no zero in $\C \setminus \Delta_1$. For $j=1,\ldots, m,$ $\mathcal{A}_{\vec n,j}$ has exactly $\eta_{\vec n,j}$ zeros in $\C \setminus \Delta_{j+1},\, (\Delta_{m+1} = \emptyset)$, they are all simple and lie in $\mathring{\Delta}_j$. If $w_{\vec n,j}, j=1,\ldots,m-1,$ denotes the monic  polynomial whose roots are the simple zeros which $\mathcal{A}_{\vec n,j}$ has in $\mathring{\Delta}_j$ then
        \begin{equation}
            \label{forma*}
        	\frac{\mathcal{A}_{\vec{n},j}}{w_{\vec{n},j}} = \mathcal{O}\left(\frac{1}{z^{\eta_{\vec{n},j+1}+1}}\right) \in \mathcal{H}(\C\setminus\Delta_{j+1}), \qquad z\to \infty.
        \end{equation}
        For each $j = 0,\ldots,m-1$ the order of interpolation at infinity prescribed in \eqref{Problem} is exact.
    \end{lemma}

    \begin{proof}
        From \eqref{Problem} for $j=0$, using Lemma \ref{lm:rem} with $w \equiv 1$,  we obtain
        \begin{equation*}
            \int x^\nu \mathcal{A}_{\vec{n},1}(x)\D \sigma_1(x) =0,\qquad \nu = 0,1,\ldots,n_1-1.
        \end{equation*}
        Therefore, $\mathcal{A}_{\vec{n},1}$ has at least $n_1$ sign changes in $\mathring{\Delta}_1$.

        Let $w_{\vec{n},1}$ be a polynomial whose roots lie in $\C \setminus \Delta_2$ and contain all the points where $\mathcal{A}_{\vec{n},1}$ changes sign in $\mathring{\Delta}_1$. Then, $\deg w_{\vec{n},1} \geq n_1$ and taking into account \eqref{Problem} for $j=1$, we obtain
        \begin{equation*}
            \frac{\mathcal{A}_{\vec{n},1}(z)}{w_{\vec{n},1}(z)} = \mathcal{O}\left(\frac{1}{z^{\eta_{\vec{n},2}+1}}\right) \in \mathcal{H}(\C\setminus\Delta_2).
        \end{equation*}

        Notice that $\mathcal{A}_{\vec{n},1}$ and $w_{\vec{n},1}$ satisfy the hypothesis of Lemma \ref{lm:rem}, so
        \begin{equation*}
            \int x^\nu \mathcal{A}_{\vec{n},2}(x)\frac{\D\sigma_2(x)}{w_{\vec{n},1}(x)}=0,\qquad \nu = 0,1,\ldots, n_1+n_2-1.\nonumber
        \end{equation*}
        This implies that $\mathcal{A}_{\vec{n},2}$ has, at least, $n_1+n_2$ sign changes in $\mathring{\Delta}_2$.

        Let $w_{\vec{n},2}$ be a polynomial whose roots lie in $\C \setminus \Delta_3$ and contain all the points where $\mathcal{A}_{\vec{n},2}$ changes sign in $\mathring{\Delta}_2$. Then, $\deg w_{\vec{n},2} \geq n_1 + n_2$ and taking into account \eqref{Problem} for $j=2$, we obtain
        \begin{equation*}
           \frac{\mathcal{A}_{\vec{n},2}(z)}{w_{\vec{n},2}(z)} = \mathcal{O}\left(\frac{1}{z^{\eta_{\vec{n},3}+1}}\right) \in  \mathcal{H}(\C\setminus\Delta_3).
        \end{equation*}
        We have deduced analogous conclusions for $\mathcal{A}_{\vec{n},2}$ as  we had for $\mathcal{A}_{\vec{n},1}$.

        We can repeat these arguments inductively and obtain that for each $j=0,\ldots,m-1$ there exists a polynomial  $w_{\vec{n},j},  \deg w_{\vec{n},j}\geq n_1+\cdots+n_{j} = \eta_{\n,j}$ ($w_{\vec n,0} \equiv 1$) whose
        roots lie in $\C \setminus \Delta_{j+1}$ and contain all the points where  $\mathcal{A}_{\vec{n},j}$ changes sign in $ \mathring{\Delta}_{j}$ and \eqref{forma*} takes place.

        For $j=m-1$, we get
        \begin{equation}
            \label{HP_s_mm}
            \frac{a_{\vec n,m}\widehat{s}_{m,m}-a_{\vec n,m-1}}{w_{\vec{n},m-1}}(z) = \mathcal{O}\left(\frac{1}{z^{\eta_{\vec{n},m}+1}}\right) \in \mathcal{H}(\C\setminus\Delta_m)
        \end{equation}
        and using again Lemma \ref{lm:rem}
        \begin{equation*}
            \int x^\nu a_{\vec n,m}(x)\frac{\D s_{m,m}(x)}{w_{\vec{n},m-1}(x)}=0,\qquad \nu =0,1,\ldots,|\vec n|-1.
        \end{equation*}
        This implies that $a_{\vec n,m}$ has at least $|\vec n|$ sign changes in $\mathring{\Delta}_m$. Since $\deg a_{\vec n, m} \leq |\vec n|$ we get that $a_{\vec n,m}$ is either identically equal to zero or it has exactly $|\vec n|$ simple zeros, all in $\mathring{\Delta}_m$. The first situation cannot occur since from \eqref{Problem} it would follow that $(a_{\vec n,0} \ldots,a_{\vec n,m})\equiv \vec{0}$. So, only the second statement is possible.

        Notice that if $\mathcal{A}_{\vec n,0}$ has  a zero in $\C \setminus \Delta_1$, or  for some $j=1,\ldots,m-1$, $\mathcal{A}_{\vec n,j}$ has more than $\eta_{\vec{n},j}$ zeros in $\C \setminus \Delta_{j+1}$, we can get an extra order of interpolation in \eqref{HP_s_mm}. This also occurs if for some $j=0,\ldots,m-1$ the order of intepolation at $\infty$ in \eqref{Problem} is higher than the one imposed.  This entails one more orthogonality relation for $a_{\vec n,m}$ implying that this polynomials is identically equal to zero which is not possible. The statements of the lemma readily follow.
    \end{proof}

    In order to prove Theorem \ref{Convergence}, we need relations similar to \eqref{HP_s_mm} for $a_{\vec n,m}\widehat{s}_{m,j+1}-a_{\vec n,j}$, $j=0,\ldots,m-1$. For this purpose, some transformations involving reciprocals and ratios of Cauchy transforms of measures will be employed. We introduce them next.

    It is well known that for each measure $\sigma\in\mathcal{M}(\Delta)$, where $\Delta \subset \R$ is either a closed bounded interval or it is a half-line (that is, it is unbounded only on one side of the real line), there exist a measure $\tau\in\mathcal{M}(\Delta)$ and a polynomial $\ell(z) = az+b$, $a=1/|\sigma|,b\in\R$, such that
    \begin{equation}
        \label{CT:inv}
        \frac{1}{\widehat{\sigma}(z)} = \ell(z) + \widehat{\tau}(z),
    \end{equation}
    where $|\sigma|$ denotes the total mass of the measure $\sigma$. For more details, in the case of measures with compact support see \cite[Appendix]{Krein_Nudelman} and \cite[Theorem 6.3.5]{Stahl_Totik}, when the measure is supported in a half line see \cite[Lemma 2.3]{ulises_lago_2}. If $\sigma$ satisfies Carleman's condition then $\tau$ also verifies that condition, \cite[Theorem 1.5]{Lago_Sergio}. We call $\tau$ the inverse measure of $\sigma$.

    Inverse measures appear frequently in our developments, so we will fix a notation to differentiate them. In relation with the measures denoted with $s$ they will carry over to them the corresponding sub-indices. The same goes for the polynomials $\ell$. For example:
    \begin{equation*}
        \frac{1}{\widehat{s}_{j,k}(z)} = \ell_{j,k}(z) + \widehat{\tau}_{j,k}(z).
    \end{equation*}
    We also use
    \begin{equation*}
        \frac{1}{\widehat{\sigma}_\alpha(z)} = \ell_\alpha(z) + \widehat{\tau}_\alpha(z).
    \end{equation*}
    On some occasions, we write $\langle \sigma_\alpha,\sigma_\beta \widehat{\rangle}$ in place of $\widehat{s}_{\alpha,\beta}$. In   \cite[Lemma 2.10]{ulises_lago_2} several formulas involving Cauchy transforms were proved. For our reasonings, the most important ones establish that
    \begin{equation}
        \label{CT:ratio}
        \frac{\widehat{s}_{1,k}}{\widehat{s}_{1,1}} = \frac{|s_{1,k}|}{|s_{1,1}|} - \langle \tau_{1,1}, \langle s_{2,k}, \sigma_1 \rangle \widehat{\rangle}.
    \end{equation}

    We also state a formula which connects forms of different levels of Nikishin systems. A proof appears in \cite[Lemma 2.1]{Lago_Sergio_Jacek}. Consider the linear forms with polynomial coefficients
    \begin{equation*}
        \mathcal{L}_j:=a_j + \sum_{k=j+1}^m a_k\Hat{s}_{j+1,k},\qquad j=0,\ldots,m-1, \qquad \mathcal{L}_m=a_m
    \end{equation*}
    where $a_j$ are arbitrary polynomials.
    \begin{lemma}
        \label{lm:levels}
        Let $(s_{1,1},\ldots,s_{1,m}) = \mathcal{N}(\sigma_1,\ldots,\sigma_m)$ be given. Then, for each $j=0,\ldots,m-2$, and $r=j+1,\ldots,m-1$
        \begin{equation}
            \label{levels}
            \mathcal{L}_j + \sum_{k=j+1}^r\widehat{s}_{k,j+1}\mathcal{L}_{k} = a_j + (-1)^{r-j}\sum_{k=r+1}^m a_k\langle s_{r+1,k}, s_{r,j+1}\widehat{\rangle}.
        \end{equation}
    \end{lemma}

    Given $\vec n = (n_1,\ldots,n_{m}) \in \Z_+^m$, set
    \begin{align*}
	    \chi_{\vec n,j,k} &:=\min\{n_j+1,n_{j+1}+2,\ldots, n_k+2\}, \quad j< k.
    \end{align*}
We are ready to prove

    \begin{lemma}
        \label{lm:asym_MP}
        Given $\vec n\in(\Z_+^m)^*$, let $a_{\vec n,0}, a_{\vec n,1},\ldots, a_{\vec n,m}$ be the Hermite-Pad\'e polynomials associated with the Nikishin system $\mathcal{N}(\sigma_1,\ldots,\sigma_m)$ such that \eqref{Problem} holds.
        Then for each $j=0,\ldots, m-2$
        \begin{equation}
            \label{orden}
            \frac{a_{\vec n,j}-a_{\vec n,m}\widehat{s}_{m,j+1}}{w^*_{\vec{n},j}}(z) = \mathcal{O}\left(z^{-\left( \eta_{\vec{n},j+1}+\sum_{k=1}^{m-j-1}\chi_{\vec n,j+1,j+k+1}-2m+2j+3\right)}\right) \in\mathcal{H}(\C\setminus\Delta_m), \quad z\rightarrow\infty,
        \end{equation}
        where $w^*_{\vec{n},j}$ is a monic polynomial with real coefficients of degree $\sum_{k=1}^{m-j-2}\chi_{\vec n,j+1,j+k+1} +\eta_{\vec{n},j+1}-2m+2j+3$ (the sum is empty when $j=m-2$). The polynomial $a_{\vec n,j}, j=0,\ldots,m-2,$ has at least $\sum_{k=1}^{m-j-1}\chi_{\vec n,j+1,j+k+1} +\eta_{\vec{n},j+1}-2m+2j+1$ sign changes in $\mathring{\Delta}_m$.
    \end{lemma}

    \begin{proof}
     Fix $j \in \{0,\ldots,m-2\}$, using  \eqref{CT:inv} and \eqref{CT:ratio}, we have
        \begin{multline*}
            \frac{\mathcal{A}_{\vec{n},j}}{\widehat{\sigma}_{j+1}} = \left((-1)^j \ell_{j+1}a_{\vec n,j} + \sum_{k=j+1}^{m}(-1)^k \frac{|s_{j+1,k}|}{|\sigma_{j+1}|}a_{\vec n,k}\right)\\ + (-1)^j a_{\vec n,j}\widehat{\tau}_{j+1} - \sum_{k=j+2}^m (-1)^k a_{\vec n,k}\langle \tau_{j+1},\langle s_{j+2,k},\sigma_{j+1}\rangle\widehat{\rangle}.
        \end{multline*}
        The quotient $\frac{\mathcal{A}_{\vec{n},j}}{\widehat{\sigma}_{j+1}}$ has the same structure as $\mathcal{A}$ in Lemma \ref{lm:rem}. Moreover, from \eqref{forma*},
        \begin{equation*}
              \frac{\mathcal{A}_{\vec{n},j}(z)}{(\widehat{\sigma}_{j+1}w_{\vec{n},j})(z)} = \mathcal{O}\left(\frac{1}{z^{\eta_{\vec{n},j+1}}}\right) \in \mathcal{H}(\C\setminus\Delta_{j+1})
        \end{equation*}
        and, as a consequence of \eqref{A:orth}, for $\nu=0,\ldots,\eta_{\vec{n},j+1}-2$, we obtain the orthogonality relations
        \begin{equation*}
            0 = \int_{\Delta_{j+1}} x^\nu \left((-1)^j a_{\vec n,j} - \sum_{k=j+2}^m (-1)^k a_{\vec n,k}\langle s_{j+2,k},\sigma_{j+1}\widehat{\rangle}\right)(x) \frac{\D \tau_{j+1}(x)}{w_{\vec{n},j}(x)}.
        \end{equation*}
        Therefore, the  expression in parenthesis under the integral sign has at least $\eta_{\vec{n},j+1}-1$ sign changes in $\mathring{\Delta}_{j+1}$. Thus, there exists a polynomial $w_{\vec{n},j,1}$ of degree $\eta_{\vec{n},j+1}-1$ whose zeros are simple and lie in $\mathring{\Delta}_{j+1}$ such that
        $$\frac{1}{w_{\vec{n},j,1}}\left((-1)^j a_{\vec n,j} - \sum_{k=j+2}^m (-1)^k a_{\vec n,k}\langle s_{j+2,k},\sigma_{j+1}\widehat{\rangle}\right)\in\mathcal{H}(\C\setminus\Delta_{j+2}).$$

        We can use Lemma \ref{lm:levels} choosing $r=j+1$ and obtain
        \begin{equation*}
            \mathcal{A}_{\vec{n},j}-\widehat{s}_{j+1,j+1}\mathcal{A}_{\vec{n},j+1} = (-1)^j a_{\vec n,j} - \sum_{k=j+2}^m (-1)^k a_{\vec n,k}\langle s_{j+2,k},\sigma_{j+1}\widehat{\rangle}.
        \end{equation*}
        From \eqref{Problem} we know that $\mathcal{A}_{\vec{n},j}-\widehat{s}_{j+1,j+1}\mathcal{A}_{\vec{n},j+1}$ is $\mathcal{O}\left(z^{-\min\{n_{j+1}+1, n_{j+2}+2\}}\right)$, $z\to \infty$. Hence,
        \begin{equation*}
            \frac{1}{w_{\vec{n},j,1}(z)}\left((-1)^j a_{\vec n,j} - \sum_{k=j+2}^m (-1)^k a_{\vec n,k}\langle s_{j+2,k},\sigma_{j+1}\widehat{\rangle}\right)(z) = \mathcal{O}\left(\frac{1}{z^{\chi_{\vec n,j+1,j+2}+\eta_{\vec{n},j+1}-1}}\right).
        \end{equation*}
        Notice that if $j=m-2$ we obtain
        \begin{equation*}
            \frac{a_{\vec n,m-2}-a_{\vec n,m}\widehat{s}_{m,m-1}}{w_{\vec{n},m-2,1}}(z) = \mathcal{O}\left(\frac{1}{z^{\chi_{\vec n,m-1,m}+\eta_{\n,m-1}-1}}\right)
        \end{equation*}
        which is \eqref{orden} for this value of $j$ taking $w^*_{\vec n,m-2} = w_{\vec n,m-2,1} $.

        Using the identity $\langle s_{j+2,k}, s_{j+1,j+1} \rangle = \langle s_{j+2,j+1},s_{j+3,k}\rangle$ for $k=j+3,\ldots,m$, we deduce
        \begin{multline}
            (-1)^j a_{\vec n,j} - \sum_{k=j+2}^m (-1)^k a_{\vec n,k}\langle s_{j+2,k},\sigma_{j+1}\widehat{\rangle}\\ = (-1)^{j}a_{\vec n,j} - (-1)^{j+2}a_{\vec n,j+2}\widehat{s}_{j+2,j+1} - \sum_{k=j+3}^m (-1)^k a_{\vec n,k} \langle s_{j+2,j+1}, s_{j+3,k}\widehat{\rangle}.\label{Prep_Elim}
        \end{multline}
        We wish to eliminate  the term with $\widehat{s}_{j+2,j+1}$ from the right hand side of \eqref{Prep_Elim}; therefore, we divide both sides of \eqref{Prep_Elim} by $\widehat{s}_{j+2,j+1}$ and use again \eqref{CT:inv} and \eqref{CT:ratio}. The right hand side becomes
        \begin{multline*}
            \left( (-1)^j a_{\vec n,j}\ell_{j+2,j+1}-(-1)^{j+2}a_{\vec n,j+2} - \sum_{k=j+3}^m (-1)^k\frac{|\langle s_{j+2,j+1},s_{j+3,k} \rangle|}{|s_{j+2,j+1}|}a_{\vec n,k}\right) +\\
            (-1)^j a_{\vec n,j}\widehat{\tau}_{j+2,j+1} + \sum_{k=j+3}^m (-1)^k a_{\vec n,k}\langle \tau_{j+2,j+1}, \langle s_{j+3,k}, s_{j+2,j+1}\rangle\widehat{\rangle}
        \end{multline*}
        which is a linear form like $\mathcal{A}$ in Lemma \ref{lm:rem}, and
        \begin{multline*}
            \frac{1}{(w_{\vec{n},j,1}\widehat{s}_{j+2,j+1})(z)}\left((-1)^j a_{\vec n,j} - \sum_{k=j+2}^m (-1)^k a_{\vec n,k}\langle  s_{j+2,k}, \sigma_{j+1}\widehat{\rangle}\right)(z) =\\
            \mathcal{O}\left(\frac{1}{z^{\chi_{\vec n,j+1,j+2}+\eta_{\vec{n},j+1}-2}}\right) \in \mathcal{H}(\C\setminus\Delta_{j+2}).
        \end{multline*}
        Therefore, for $\nu=0,1,\ldots, \chi_{\vec n,j+1,j+2}+\eta_{\vec{n},j+1}-4$,
        \begin{equation*}
            \int x^\nu \left((-1)^j a_{\vec n,j} + \sum_{k=j+3}^m (-1)^k a_{\vec n,k}\langle  s_{j+3,k}, s_{j+2,j+1}\widehat{\rangle}\right)(x)\frac{\D \tau_{j+2,j+1}(x)}{w_{\vec{n},j.1}(x)} =0.
        \end{equation*}
        So, the expression in parenthesis under the integral sign has at least $\chi_{\vec n, j+1,j+2}+\eta_{\vec{n},j+1}-3$ sign changes in $\mathring{\Delta}_{j+2}$, and we can guarantee the existence of a polynomial $w_{\vec{n},j,2}$, $\deg w_{\vec{n},j,2}=\chi_{\vec n,j+1,j+2}+\eta_{\vec{n},j+1}-3$, with simple zeros located   inside $\Delta_{j+2}$ such that
        \begin{equation*}
            \frac{1}{w_{\vec{n},j,2}}\left((-1)^j a_{\vec n,j} + \sum_{k=j+3}^m (-1)^k a_{\vec n,k}\langle  s_{j+3,k}, s_{j+2,j+1}\widehat{\rangle}\right)\in\mathcal{H}(\C\setminus\Delta_{j+3}).
        \end{equation*}
        On the other hand, using Lemma \ref{lm:levels} with $r=j+2$ and the definition of ML Hermite-Pad\'e approximant, we get
        \begin{multline*}
             \mathcal{A}_{\vec{n},j} - \widehat{s}_{j+1,j+1}\mathcal{A}_{\vec{n},j+1} + \widehat{s}_{j+1,j+1}\mathcal{A}_{\vec{n},j+2} = \\
             (-1)^j a_{\vec n,j} + \sum_{k=j+3}^m (-1)^k a_{\vec n,k}\langle  s_{j+3,k}, s_{j+2,j+1}\widehat{\rangle} \in \mathcal{O}\left(\frac{1}{z^{\chi_{\n,j+1,j+3}}}\right).
        \end{multline*}
        Thus,
        \begin{equation*}
            \frac{1}{w_{\vec{n},j,2}}\left((-1)^j a_{\vec n,j} + \sum_{k=j+3}^m (-1)^k a_{\vec n,k}\langle  s_{j+3,k}, s_{j+2,j+1}\widehat{\rangle}\right)\in\mathcal{O}\left(\frac{1}{z^{\chi_{\vec n,j+1,j+3}+\chi_{\vec n,j+1,j+2}+\eta_{\vec{n},j+1}-3}}\right).
        \end{equation*}
        In particular, if $j=m-3$, we get
        \begin{equation*}
            \frac{(a_{\vec n,m-3}-a_{\vec n,m}\widehat{s}_{m,m-2})(z)}{w_{\vec{n},m-3,2}(z)}=\mathcal{O}\left(\frac{1}{z^{\chi_{\vec n,m-2,m}+\chi_{\vec n,m-2,m-1}+\eta_{m-2}-3}}\right),
        \end{equation*}
        which gives us \eqref{orden} with $w^*_{\n,m-3} = w_{\vec{n},m-3,2}$ when $j=m-3$.

        This process can be continued inductively, and after $m-j-1$ reductions we guarantee the existence of a polynomial $w_{\vec{n},j,m-j-1}$ of degree $\sum_{k=1}^{m-j-2}\chi_{\vec n,j+1,j+k+1} +\eta_{\vec{n},j+1}-2m+2j+3$ with simple zeros in  $\mathring{\Delta}_{m-1}$ such that
        \begin{equation*}
            \frac{a_{\vec n,j}-a_{\vec n,m}\widehat{s}_{m,j+1}}{w_{\vec{n},j,m-j-1}}(z) = \mathcal{O}\left(z^{-\left(\sum_{k=1}^{m-j-1}\chi_{\vec n,j+1,j+k+1} +\eta_{\vec{n},j+1}-2m+2j+3\right)}\right) \in\mathcal{H}(\C\setminus\Delta_m), \quad z\rightarrow\infty,
        \end{equation*}
        which allows us to deduce \eqref{orden} taking $w^*_{\n,j} = w_{\vec{n},j,m-j-1}$.

        As an immediate consequence we have
        \begin{equation*}
            \frac{a_{\vec n,j}-a_{\vec n,m}\widehat{s}_{m,j+1}}{\widehat{s}_{m,j+1}w^*_{\vec{n},j}}(z) =  \mathcal{O}\left(z^{-\left(\sum_{k=1}^{m-j-1}\chi_{\vec n,j+1,j+k+1} +\eta_{\vec{n},j+1}-2m+2j+2\right)}\right) \in\mathcal{H}(\C\setminus\Delta_m), \qquad z\rightarrow\infty,
        \end{equation*}
        but
        \begin{equation*}
            \frac{a_{\vec n,j}-a_{\vec n,m}\widehat{s}_{m,j+1}}{\widehat{s}_{m,j+1}} = a_{\vec n,j}\widehat{\tau}_{m,j+1} - (a_{\vec n,m}-\ell_{m,j+1}a_{\vec n,j}).
        \end{equation*}
        Hence,
        \begin{equation*}
            \int x^\nu a_{\vec n,j}(x)\frac{\D\tau_{m,j+1}(x)}{w^*_{\vec{n},j}(x)} =0, \qquad \nu=0,1,\ldots,\sum_{k=1}^{m-j-1}\chi_{\vec n,j,j+k} +\eta_{\vec{n},j}-2m+2j,
        \end{equation*}
        and the polynomial $a_{\vec n,j}$ has at least $\sum_{k=1}^{m-j-1}\chi_{\vec n,j+1,j+k+1} +\eta_{\vec{n},j+1}-2m+2j+1$ sign changes in $\mathring{\Delta}_m$ which is the last statement of the lemma.
    \end{proof}

\subsection{Proof of Theorem \ref{Convergence}}
    Let us begin with the simplest case when $j=m-1$. Let $\Lambda \subset (\Z_+^m)^*$ be an arbitrary sequence of multi-indices. According to \eqref{forma*} (recall that $n_1 + \cdots +n_{m} = |\n|$)
    \[
    \frac{a_{\vec n,m-1}-a_{\vec n,m}\widehat{s}_{m,m}}{w_{\vec{n},m-1}}(z) =
    \mathcal{O}\left(\frac{1}{z^{\left( |\n| + 1\right)}}\right) \in\mathcal{H}(\C\setminus\Delta_m), \quad z\rightarrow\infty,
    \]
    and $\deg w_{\vec{n},m-1} = \eta_{\n,m-1} < 2|\vec n|$. Since $\deg a_{\vec n,m} \leq |\vec n|, \deg a_{\vec n,m-1} \leq |\vec n|-1$, It follows that $a_{\vec n,m-1}/a_{\vec n,m}$ is the standard multipoint Pad\'e approximant of $\widehat{s}_{m,m}$ with respect to $w_{\vec{n},m-1}$ (see \cite{lago79}). This implies that $a_{\n,m}$ is the $|\n|$-th monic orthogonal polynomial with respect to $\D s_{m,m}/w_{\vec{n},m-1}$ and $a_{\n,m-1}$ the corresponding polynomial of second kind. This implies that the zeros of these polynomials lie in $\mathring{\Delta}_m$ and interlace. Now, \eqref{markov} for $j=m-1$ readily follows from \cite[Theorem 1]{lago79} (see \cite[Corollary 1]{lago79})  in the case that the sequence of moments of $\sigma_m$ verifies Carleman's condition. When $\Delta_{m-1}$ is a compact interval bounded away from $\Delta_m$ and $\lim_{\vec n \in \Lambda} \eta_{\vec n,m-1} = \infty$ then the number of interpolation conditions on $\Delta_{m-1}$ (at the zeros of $w_{\vec n,m-1}$) suffice to guarantee the convergence of the sequence of approximants, which follows from (\cite[Theorem 1, Corollary 2]{lago79}).

    For other values of $j$, there is some defect in the order of interpolation on the right hand of $\eqref{orden}$ and we cannot ensure that $a_{\vec n,j}/a_{\vec n,m}$ is an $|\n|$-th multipoint Pad\'e approximant. That is the reason for restricting the sequence of multi-indices in the first part of the statement of Theorem \ref{Convergence}.

    In the sequel, we assume that $\Lambda \subset (\Z_+^{m})^*$ is an infinite sequence of distinct multi-indices such that there exist $\ell\in \{0,\ldots,m-2\}$ and a   non-negative integer $N$ such that $n_{j+1} \leq n_{j} + N$  for all $\ell+1 \leq j \leq m-1$. In this case, we automaticaly have that $\lim_{\vec n \in \Lambda} \eta_{\vec n,m-1} = +\infty$. Indeed, assume that $\limsup_{\vec n \in \Lambda} \eta_{\vec n,m-1} < +\infty$. In particular, this implies that there exists a constant $C$ such that $n_{m-1} \leq C, \vec n \in \Lambda$. However, $\lim_{\vec n \in \Lambda} n_m = +\infty$ because $\lim_{\vec n \in \Lambda}|\vec n| = \infty$ since the multi-indices are distinct; therefore, it is impossible that $n_{m} \leq n_{m-1} + N, \vec n \in \Lambda$.

    For $j=m-1$ the proof of \eqref{markov} was carried out above. Fix $j\in \{\ell,\ldots,m-2\}$. We have
    \begin{align*}
        \deg{w_{\n,j}^*} =& \sum_{k=1}^{m-j-2}\chi_{\vec n,j+1,j+k+1} +\eta_{\vec{n},j+1}-2m+2j+3\\
                      \leq& \eta_{\vec{n},j+1} -2m+2j+3 + \sum_{k=1}^{m-j-2} (n_{j+k+1} + 2)\\
                         =& \eta_{\n,m-1} -2m+2j+3 +2(m-j-2) \leq \eta_{\n,m-1} -1 < 2|\n|
    \end{align*}
    for all $\n \in \Lambda$. Due to the assumptions imposed of the sequence $\Lambda$, we have
    \[ n_{j+k+1} \leq n_{j+k} + N \leq \cdots \leq n_{j+1} + kN.
    \]
    Therefore,
    \[ \chi_{\n,j+1,j+k+1} \geq \min\{n_{j+1},\ldots,n_{j+k+1}\} \geq  n_{j+k+1} - kN
    \]
    Consequently,
    \begin{equation}
        \label{numberzeros}
        \eta_{\vec{n},j+1}+\sum_{k=1}^{m-j-1}\chi_{\vec n,j+1,j+k+1}-2m+2j+1 \geq |\n|-2m  - N\sum_{k=1}^{m-j-1} k \geq |\n|-2m- N\frac{m(m+1)}{2}.
    \end{equation}
    Combined with the last statement of Lemma \ref{lm:asym_MP} this inequality gives the lower bound on the number of sign changes of $a_{\n,j}$ on $\mathring{\Delta}_m$.

        From \eqref{numberzeros} and \eqref{orden}, it follows that there exists a constant $\kappa \in \Z_+$  such that for all $\n \in \Lambda$ and $j = \ell,\ldots,m-2$
        \begin{equation}
            \label{multiPade}
            \frac{a_{\vec n,j}-a_{\vec n,m}\widehat{s}_{m,j+1}}{w^*_{\vec{n},j}}(z) = \mathcal{O}\left(\frac{1}{z^{|\n| +1 - \kappa}} \right) \in\mathcal{H}(\C\setminus\Delta_m), \qquad z\rightarrow\infty.
        \end{equation}
        We also have $\deg a_{\n,j} \leq |\n|-1, \deg a_{\n,m} \leq |\n|$ and $\deg w_{\n,m}^* \leq 2|\n|$. This means that for each fixed $j, \, \ell \leq j \leq m-2$, $\left(\frac{a_{\n,j}}{a_{\n.m}}\right)_{\n \in \Lambda}$ is a sequence of ``incomplete'' diagonal multipoint Pad\'e approximants of $\widehat{s}_{m.j+1}$ which satisfies \eqref{multiPade}. It is easy to verify that if the sequence of moments of $\sigma_m$ verifies Carleman's condition then for all $j, 0\leq j \leq m-1,$ the sequence of moments of $s_{m,j+1}$ also verifies Carleman's condition. Also, recall that in the present situation $\lim_{\vec n \in \Lambda} \eta_{\vec n,m-1} = +\infty$ takes place. Using the assumptions imposed on the moments of $\sigma_m$ or on $\Delta_{m-1}$ from \cite[Lemma 2]{Bust_Lago}, it follows that  $\left(\frac{a_{\n,j}}{a_{\n.m}}\right)_{\n \in \Lambda}$ converges to $\widehat{s}_{m,j+1}$ in $1$-Hausdorff content on each compact subset of $\overline{\C} \setminus \Delta_m$.
        Let us explain what convergence in $1$-Hausdorff content means.

        Let $A$ be a subset of $\C$. By $\mathcal{U}(A)$ we denote the class of all coverings of $A$ by at most a numerable set of disks. Set
        \begin{equation*}
            h(A) = \inf\left\{ \sum_{i=1}^\infty |U_i| \mid \{U_i\}\in\mathcal{U}(A)\right\},
        \end{equation*}
        where $|U_i|$ stands for the radius of the disk $U_i$. The quantity $h(A)$ is called the 1-dimensional Hausdorff content of the set $A$. Convergene in $1$-Hausdorff content means that for each compact  $K \subset \C \setminus \Delta_m$ and for each $\varepsilon>0$, we have
        \begin{equation}
            \label{convH}
            \lim_{\n \in \Lambda} h\left\{ z\in K : \left|\frac{a_{\n,j}(z)}{a_{\n.m}(z)}-\widehat{s}_{m,j+1}(z)\right|>\varepsilon \right\} =0
        \end{equation}

        The rational functions $\frac{a_{\n,j}}{a_{\n.m} }$ are holomorphic in $\C \setminus \Delta_m$ because the zeros of $a_{\n,m}$ lie in $\Delta_m$.  This together with \eqref{convH} imply that the convergence is uniform on each compact subset of $\C \setminus \Delta_m$ according to \cite[Lemma 1]{gonchar}.

        If $\Delta_m$ is bounded, we still have to consider those compact subsets of $\overline{\C} \setminus \Delta_m$ which contain $\infty$. Due to the fact that the rational functions   and $\widehat{s}_{m,j+1}$ equal zero at $\infty$ this situation is obtained from the general case using the maximum principle. The proof is complete. \hfill $\Box$

    Due to the last statement of Lemma \ref{lm:asym_MP}, \eqref{numberzeros} gives a lower bound on the number of zeros which $a_{\n,j}$ has in $\mathring{\Delta}_m$ when $\n \in \Lambda$ and $\Lambda$ verifies the conditions of Theorem \ref{Convergence}. If we impose greater restrictions on $\Lambda$ more can be said in this regard.

    \begin{theorem}
        \label{cor:comp}
        Let $\Lambda \subset (\Z_+^{m})^*$ be an infinite sequence of distinct multi-indices for which there exists $\ell\in \{0,\ldots,m-2\}$  such that $n_{j} \geq n_{j+1} + 1$  for all $\ell +1\leq j \leq m-1$ and $\vec n \in \Lambda$. Consider the sequence of vector polynomials $(a_{\vec n,0},\ldots,a_{\vec n,m})_{\vec n \in \Lambda}$ associated with  $\mathcal{N}(\sigma_1,\ldots,\sigma_m)$. Then, $a_{\n,j}, j=\ell,\ldots,m-1,$ has exactly $|\n| -1$ simple zeros which interlace the zeros of $a_{\n,m}$.
    \end{theorem}

    \begin{proof}
        We prove this by showing that for all $j = \ell,\ldots,m-1$ and $\n \in \Lambda$, the rational function $a_{\n,j}/a_{\n,m}$ is a diagonal multipoint Pad\'e approximant of $\widehat{s}_{m.j+1}$. Due to \eqref{orden} we achieve this if we show that
        \begin{equation*}
            \sum_{k=1}^{m-j-1}\chi_{\vec n,j+1,j+k+1} +\eta_{\vec{n},j+1}-2m+2j+3 = |\n|-l+1 \quad \textrm{ with } \quad l=0.
        \end{equation*}
        Notice that
        \begin{equation*}
            \sum_{k=1}^{m-j-1}\chi_{\vec n,j+1,j+k+1} +\eta_{\vec{n},j+1}-2m+2j+3 = \sum_{k=1}^{m-j-1}\chi_{\vec n,j+1,j+k+1} + |\n| -\sum_{i=j+2}^{m} n_i -2m+2j +3.
        \end{equation*}
        Combining these two relations, cancelling out common terms and making a change of parameter in the indices of the sums, we obtain the equation
        \begin{equation*}
            l= 2(m-j-1)- \sum_{k=j+2}^{m}(\chi_{\vec n,j+1,k}-n_{k}).
        \end{equation*}
        Taking into account that $n_j\geq n_{j+1}+1,\ell +1\leq j \leq m-1$, it readily follows that $\chi_{\vec n,j+1,k}=n_k+2$. Consequently,
        $$\sum_{k=j+2}^{m}(\chi_{\vec n,j+1,k}-n_{k}) = 2(m-j-1)$$
        and thus $l=0$ as needed.

        Therefore,
        \begin{equation*}
            \frac{(a_{\vec n,j}-a_{\vec n,m}\widehat{s}_{m,j+1})(z)}{w^*_{\vec{n},j}(z)} = \mathcal{O}\left(\frac{1}{z^{|\vec n|+1}}\right),\qquad j=\ell,\ldots,m-1.
        \end{equation*}
        and $\deg w^*_{\vec{n},j} \leq 2|\n|$.  Consequently, $a_{\n,j}/a_{\n,m}$ is the $|\n|$-th diagonal multipoint Pad\'e approximant with respect to $\widehat{s}_{m,j+1}$ with interpolation points at the zeros of $w^*_{\vec{n},j}$, and at $\infty$ of order $2|\n| -\deg w^*_{\vec{n},j}$.

        So, the fraction $a_{\vec n,j}/a_{\vec n,m}$ is the $|\vec n|$-th diagonal multipoint Pad\'e approximation of $\widehat{s}_{m,j+1}$. From the theory of diagonal multipoint Pad\'e approximation (or simply using \eqref{A:orth}) we know that $a_{\vec n,m}$  is the $|\n|-th$ monic orthogonal polynomials with respect to the varying measure $\D s_{m, j+1}/w^*_{\n,j}$ and $a_{\n,j}$ is the corresponding polynomial of the second kind whose zeros interlace those of $a_{\vec n,m}$. We are done.
    \end{proof}

\section{Ratio asymptotic}\label{sec:ratio}

    Throughout this section $Q_{\n,j}$, $j=1,\ldots,m,$ denotes the monic polynomial whose roots coincide with the zeros of $\mathcal{A}_{\n,j}$ in $\C \setminus \Delta_{j+1}$ ($\Delta_{m+1} = \emptyset$). In Lemma \ref{lm:zeros}, these polynomials were denoted $w_{\n,j}$.  From that lemma it follows that  $\deg Q_{\n,j} = \eta_{\n,j}= n_1 + \cdots + n_{j}$, its zeros are simple and lie in $\mathring{\Delta}_{j}$.  We will show that these polynomials satisfy full orthogonality relations with respect to certain varying measures. This fact plays an important role in the study of ratio asymptotic.

\subsection{Multi-orthogonality relations}

    From Lemma \ref{lm:rem} and \eqref{forma*} in  Lemma \ref{lm:zeros} it readily follows that  for $j=0,\ldots,m-1$
    \begin{equation}
        \label{int_Anj}
        \frac{\mathcal{A}_{\vec{n},j}(z)}{Q_{\n,j}(z)} = \int \frac{\mathcal{A}_{\vec{n},j+1}(x)}{z-x}\frac{\D\sigma_{j+1}(x)}{Q_{\n,j}(x)},
    \end{equation}
    where $Q_{\n,0}\equiv 1$, and
    \begin{equation}
        \label{orth_Anj}
        \int x^\nu\mathcal{A}_{\vec{n},j+1}(x)\frac{\D\sigma_{j+1}(x)}{Q_{\n,j}(x)} = 0,\qquad \nu =0,1,\ldots \eta_{\vec{n},j+1}-1.
    \end{equation}

    Set
    \begin{equation}
        \label{Hnj}
        \mathcal{H}_{\n,j}:= \frac{Q_{\n,j+1}\mathcal{A}_{\vec{n},j}}{Q_{\n,j}},\qquad j=0,\ldots,m-1,
    \end{equation}
    where $Q_{\n,0} \equiv Q_{\n,m+1} \equiv 1$. Since $\mathcal{A}_{\n,m} = (-1)^m a_{\n,m}$ and $a_{\n,m}$ is monic, we take $ \mathcal{H}_{\n,m} = (-1)^m$.

    \begin{lemma}
        \label{mult_orth}
        Consider the Nikishin system $\mathcal{N}(\sigma_1,\ldots,\sigma_m)$. For each fixed $\n\in(\Z_+^m)^*$ and $j=0,\ldots,m-1$
        \begin{equation}
            \label{orth_Qnj}
            \int x^\nu Q_{\n,j+1}(x)\frac{\mathcal{H}_{\n,j+1}(x)\D\sigma_{j+1}(x)}{Q_{\n,j}(x)Q_{\n,j+2}(x)} =0, \qquad \nu = 0,\ldots, \eta_{\vec{n},j+1}-1
        \end{equation}
        and
        \begin{equation}
        \label{Hnj_int}
        \mathcal{H}_{\n,j}(z) = \int \frac{Q_{\n,j+1}^2(x)}{z-x}\frac{\mathcal{H}_{\n,j+1}(x)\D\sigma_{j+1}(x)}{Q_{\n,j}(x)Q_{\n,j+2}(x)}.
        \end{equation}
    \end{lemma}

    \begin{proof}
        Formula \eqref{orth_Qnj} is \eqref{orth_Anj} rewritten with the new notation. Since $\deg Q_{\n,j+1}=\eta_{\vec{n},j+1}$, \eqref{orth_Qnj} implies that
        \begin{equation*}
            \int \frac{Q_{\n,j+1}(z) - Q_{\n,j+1}(x)}{z-x}Q_{\n,j+1}(x)\frac{\mathcal{H}_{\n,j+1}(x)\D\sigma_{j+1}(x)}{Q_{\n,j}(x)Q_{\n,j+2}(x)} =0.
        \end{equation*}
        Consequently,
        \begin{equation*}
            Q_{\n,j+1}(z)\int \frac{Q_{\n,j+1}(x)}{z-x}\frac{\mathcal{H}_{\n,j+1}(x)\D\sigma_{j+1}(x)}{Q_{\n,j}(x)Q_{\n,j+2}(x)} =
            \int \frac{Q_{\n,j+1}^2(x)}{z-x}\frac{\mathcal{H}_{\n,j+1}(x)\D\sigma_{j+1}(x)}{Q_{\n,j}(x)Q_{\n,j+2}(x)}.
        \end{equation*}
        Taking into account \eqref{Hnj} and \eqref{int_Anj} we get
        \begin{equation*}
            \int \frac{Q_{\n,j+1}(x)}{z-x}\frac{\mathcal{H}_{\n,j+1}(x)\D\sigma_{j+1}(x)}{Q_{\n,j}(x)Q_{\n,j+2}(x)} = \int\frac{\mathcal{A}_{\vec{n},j+1}(x)}{z-x}\frac{\D\sigma_{j+1}(x)}{Q_{\n,j}(x)} = \frac{\mathcal{A}_{\vec{n},j}(z)}{Q_{\n,j}(z)}.
        \end{equation*}
        Therefore, \eqref{Hnj_int} holds.
    \end{proof}

    Given $\n \in (\Z_+^*)^*$ and $l \in\{1,\ldots,m\}$, by $\n^l$ we denote the multi-index obtained adding $1$ to the $l$-th component of $\n$. In the next lemma, we  prove that the zeros of the polynomials $Q_{\n,j}$ and $Q_{\n^l,j}$ interlace.  The idea of the proof was borrowed from \cite[Theorem 2.1]{ALR}.

    \begin{lemma}
        \label{lm:interlace}
        Consider the Nikishin system $\mathcal{N}(\sigma_1,\ldots,\sigma_m)$. For each  $\vec{n}\in(\Z_+^m)^*$ and $j=1,\ldots,m$, the zeros of the forms $\mathcal{A}_{\vec{n},j}$ and $\mathcal{A}_{\n^l,j}$ in $\mathring{\Delta}_j$ interlace.
    \end{lemma}
    \begin{proof}
        Fix $\vec{n}\in\Z_+^m$ and $j \in \{1,\ldots,m\}$. Let $\alpha,\beta\in\R$ be such that $\alpha^2+\beta^2\neq 0$. Define the linear form
        \begin{equation*}
            \mathcal{D}_{\vec{n},j} := \alpha\mathcal{A}_{\vec{n},j} + \beta\mathcal{A}_{\n^l,j}.
        \end{equation*}

        Repeating the arguments in the proof of Lemma \ref{lm:zeros} we deduce that  the form $\mathcal{D}_{\vec{n},j}$ has at least $\eta_{\vec{n},j}$ sign changes in  $\mathring{\Delta}_j$, and at most $\eta_{\vec{n},j}+1$ zeros in $\C\setminus\Delta_{j+1}$ ($\Delta_{m+1}=\emptyset$). Consequently, all the zeros of $\mathcal{D}_{\vec{n},j}$ in $\C\setminus\Delta_{j+1}$ are real and simple.

        From this assertion, we deduce that the forms $\mathcal{A}_{\vec{n},j}$ and $\mathcal{A}_{\n^l,j}$ cannot have common zeros. If such a point $y$   exists, the function
        \begin{equation*}
            \mathcal{D}_{\vec{n},j}(x) = \mathcal{A}_{\vec{n},j}(x) - \frac{\mathcal{A}_{\vec{n},j}'(y)}{\mathcal{A}_{\n^l,j}'(y)}\mathcal{A}_{\n^l,j}(x)
        \end{equation*}
        would have a double zero at $y$. But this last statement contradicts what we already know.

        Fix $y\in\R\setminus\Delta_{j+1}$, and consider the form
        \begin{equation*}
            \mathcal{D}_{\vec{n},j,y}(x) = \mathcal{A}_{\n^l,j}(y)\mathcal{A}_{\vec{n},j}(x) - \mathcal{A}_{\vec{n},j}(y)\mathcal{A}_{\n^l,j}(x).
        \end{equation*}
        By construction $\mathcal{D}_{n,j,y}(y)=0$, and thus $\mathcal{D}_{n,j,y}'(y)\neq 0$. Take two consecutive zeros $y_1,y_2$ of $\mathcal{A}_{\n^l,j}$ in $\R\setminus\Delta_{j+1}$ and suppose that $y_1<y_2$. The zeros of $\mathcal{A}_{\n^l,j}$ are simple; therefore, $\mathcal{A}_{\n^l,j}'(y_1)\neq 0$ and $\mathcal{A}_{\n^l,j}'(y_2)\neq 0$. Since $\mathcal{A}_{\n^l,j}$ and $\mathcal{A}_{\n,j}$ have no common zero, we also get that $\mathcal{A}_{\vec{n},j}(y_1)\neq 0$ and $\mathcal{A}_{\vec{n},j}(y_2)\neq 0$. Thus,
        \begin{align*}
            \mathcal{D}_{\vec{n},j,y_1}'(y_1) &= -\mathcal{A}_{\vec{n},j}(y_1)\mathcal{A}_{\n^l,j}'(y_1) \neq 0,\\
            \mathcal{D}_{\vec{n},j,y_2}'(y_2) &= -\mathcal{A}_{\vec{n},j}(y_2)\mathcal{A}_{\n^l,j}'(y_2) \neq 0.
        \end{align*}
        However, the function $\mathcal{D}_{\vec{n},j,y}'(y)$ preserves the same sign all along the interval $[y_1,y_2]$. Notice that $\mathcal{A}_{\n^l,j}'(y)$ changes   sign when $y$ moves from $y_1$ to $y_2$, so  $\mathcal{A}_{\vec{n},j}$ must also change sign. By Bolzano's theorem $\mathcal{A}_{\vec{n},j}$ has a zero in $(y_1,y_2)$. The proof is complete.
    \end{proof}

\subsection{The Riemann surface}

    The ratio asymptotic of the ML multiple orthogonal polynomials is described in terms of the branches of a conformal mapping defined on a Riemann surface associated with the geometry of the problem. In the sequel, we assume that $\Delta_k$ is a closed bounded interval for all $k=1,\ldots,m$. Let us briefly describe the Riemann surface of interest.

    Let $\mathcal{R}$ denote the compact Riemann surface
    \[
    \mathcal{R}=\overline{\bigcup_{k=0}^{m}\mathcal{R}_{k}}
    \]
    formed by the $m+1$ consecutively ``glued'' sheets
    \[
    \mathcal{R}_{0}:=\overline{\mathbb{C}}\setminus\Delta_{1},\qquad \mathcal{R}_{k}:=\overline{\mathbb{C}}\setminus(\Delta_{k}\cup\Delta_{k+1}),\quad k=1,\ldots,m,\qquad \mathcal{R}_{m}:=\overline{\mathbb{C}}\setminus\Delta_{m},
    \]
    where the upper and lower banks of the slits of two neighboring sheets are identified. This surface is of genus zero. For this and other notions of Riemann surfaces as well as meromorphic functions defined on them we recommend \cite{miranda}.

    Let $\pi: \mathcal{R} \longrightarrow \overline{\mathbb{C}}$ be the canonical projection from $\mathcal{R}$ to $\overline{\mathbb{C}}$ and denote by $z^{(k)}$ the point on  $\mathcal{R}_k$ satisfying $\pi(z^{(k)}) = z$, $z \in \overline{\mathbb{C}}$. For a fixed $l\in\{1,\ldots,m\}$, let $\psi^{(l)}:\mathcal{R}\longrightarrow\overline{\mathbb{C}}$ denote a conformal mapping whose divisor consists of one simple zero at the point $\infty^{(0)}\in\mathcal{R}_{0}$ and one simple pole at  $\infty^{(l)}\in\mathcal{R}_{l}$. This mapping exists and is uniquely determined up to a multiplicative constant. Denote the branches of $\psi^{(l)}$ by
    \begin{equation}
        \label{branches}
        \psi_k^{(l)}(z) := \psi^{(l)}(z^{(k)}), \qquad k= 0,\ldots,m, \qquad z^{(k)} \in \mathcal{R}_{k}.
    \end{equation}
    From the properties of $\psi^{(l)}$, we have
    \begin{equation}
        \label{divisorcond}
        \psi_0^{(l)}(z)=C_{1,l}/z+O(1/z^{2}),\,\,\,z\rightarrow\infty,\qquad \psi_l^{(l)}(z)=C_{2,l}\,z+O(1),\,\,\,z\rightarrow\infty,
    \end{equation}
    where $C_{1,l}$, $C_{2,l}$ are non-zero constants.

    It is well known and easy to verify that the function $\prod_{k=0}^{m}\psi_{k}^{(l)}$ admits an analytic continuation to the whole extended plane $\overline{\C}$ without singularities; therefore, it is constant. Multiplying $\psi^{(l)}$ if necessary by a suitable non-zero constant, we may assume that $\psi^{(l)}$ satisfies the conditions
    \[
    \prod_{k=0}^{m}\psi_{k}^{(l)} = C, \qquad |C| = 1, \qquad C_{1,l} > 0.
    \]
    Let us show that with this normalization, $C$ is either $+1$ or $-1$.

    Indeed, for a point $z^{(k)} \in \mathcal{R}_k$ on the Riemann surface we define its conjugate $\overline{z^{(k)}} := \overline{z}^{(k)}$. Now, let $\overline{\psi}^{(l)}: \mathcal{R} \longrightarrow \overline{\mathbb{C}}$ be the function defined by $\overline{\psi}^{(l)}(\zeta):= \overline{\psi^{(l)}(\overline{\zeta})}$. It is easy to verify that $\overline{\psi}^{(l)}$ is a conformal mapping of $\mathcal{R}$ onto $\overline{\mathbb{C}}$ with the same divisor as $\psi^{(l)}$. Therefore, there exists a constant $c$ such that $\overline{\psi}^{(l)} = c \psi^{(l)}$. The corresponding branches satisfy the relations
    \[ \overline{\psi}_k^{(l)}(z) = \overline{\psi_k^{(l)}(\overline{z})} = c {\psi}_k^{(l)}(z), \qquad k=0,\ldots,m.
    \]
    Comparing the Laurent expansions at $\infty$ of $\overline{\psi_0^{(l)}(\overline{z})}$ and $c {\psi}_0^{(l)}(z)$, using the fact that $C_{1,l} >0$, it follows that $c = 1$. Then
    \[   {\psi}_k^{(l)}(z) = \overline{\psi_k^{(l)}(\overline{z})}, \qquad k=0,\ldots,m.
    \]
    This in turn implies that for each $k=0,\ldots, m,$ all the coefficients, in particular the leading one, of the Laurent expansion at infinity of $ {\psi}_k^{(l)}$ are real numbers. Obviously, $C$ is the product of these leading coefficients. Therefore, $C$ is real, and $|C|=1$ implies that $C$ equals $1$ or $-1$ as claimed. So, we can assume in the following that
    \begin{equation}
        \label{normconfmap}
        \prod_{k=0}^{m}\psi_{k}^{(l)}\equiv e,\qquad C_{1,l}>0,
    \end{equation}
    where $e$ is either $1$ or $-1$. It is easy to see that conditions \eqref{divisorcond} and \eqref{normconfmap} determine $\psi^{(l)}$ uniquely.

    We will need the following lemma. Its proof can be found in \cite[Lemma 4.2]{ALR}
    \begin{lemma}
        \label{lm:BVP}
        Set
        \begin{equation}
            \label{boundary3}
            F_{k}^{(l)}:=\prod_{\nu=k}^{m} \psi_{\nu}^{(l)}
        \end{equation}
        where the algebraic functions $\psi_{\nu}^{(l)}$ are defined by  \eqref{branches}-\eqref{normconfmap}. The collection of functions $F_{k}^{(l)}, k=$ $1, \ldots, m,$ is the unique solution of the system of boundary value problems
        \begin{align*}
            &	\text { 1) } F_{k}^{(l)}, 1 / F_{k}^{(l)} \in \mathcal{H}\left(\C \setminus \Delta_{k}\right) \\
            &		\text { 2a) } F_{k}^{(l)}(\infty)>0, \quad k=1, \ldots, l-1\\
            &	\text { 2b) }\left(F_{k}^{(l)}\right)^{\prime}(\infty)>0, \quad k=l, \ldots, m\\
            &		    	\text { 3) }\left|F_{k}^{(l)}(x)\right|^{2} \frac{1}{\left|\left(F_{k-1}^{(l)} F_{k+1}^{(l)}\right)(x)\right|}=1, \quad x \in \Delta_{k}
        \end{align*}
        where $ F_{0}^{(l)} \equiv F_{m+1}^{(l)} \equiv 1$.
    \end{lemma}

\subsection{Proof of Theorem \ref{th:ratiom}}

    Theorem \ref{th:ratiom} will be derived from Theorem \ref{Convergence} and Theorem \ref{th:ratio} below which gives the ratio asymptotic of the polynomials $Q_{\n,j}$. In proving Theorem \ref{th:ratio}, we adapt  the scheme developed in \cite[Theorem 1.2]{ALR} for the study of the ratio asymptotic of type \textsc{ii} Hermite-Pad\'e polynomials of Nikishin systems.

    Given an arbitrary function $F(z)$ which has in a neighborhood of infinity a Laurent expansion of the form $F(z)=C z^{k}+\mathcal{O}\left(z^{k-1}\right), C \neq 0,$ and $k \in \mathbb{Z},$ we denote
    \begin{equation*}
        \widetilde{F}:=\frac{F}{C}.
    \end{equation*}

    \begin{theorem}
	    \label{th:ratio}
        Consider the Nikishin system $\mathcal{N}(\sigma_1,\ldots,\sigma_m)$ where the intervals $\Delta_k$, $k=1,\ldots,m,$ are bounded and $\sigma_k' \neq 0$ a.e. in $\Delta_k$. Let $\Lambda \subset (\Z_+^m)^*$ be an infinite sequence of distinct multi-indices for which  there exists a non-negative integer $N$ such that $n_{j+1} \leq n_{j} + N$  for all $1 \leq j \leq m-1$ and $\vec n \in \Lambda$.
        Then for $k=1,\ldots, m$
	    \begin{equation}
            \label{left}
	        \lim_{\n \in \Lambda} \frac{Q_{\n^l,k}(z)}{Q_{\n,k}(z)} = \widetilde{F}^{(l)}_k(z),
	    \end{equation}
	    uniformly on each compact subset of $\C\setminus\Delta_k$.
    \end{theorem}

    {\sl Proof of Theorems \ref{th:ratio} and \ref{th:ratiom}.} From Lemma \ref{lm:interlace} we know that, for each $k=1,\ldots,m$ the zeros of $Q_{\n,k}$ and $Q_{\n^l,k}$ interlace on $\mathring{\Delta}_k$. Consequently, the family of functions $(Q_{\n^l,k}/Q_{\n,k})_{\n\in \Lambda}$ is uniformly bounded on each compact subset of $\C\setminus\Delta_k$. Therefore, there exists $\Lambda' \subset \Lambda$ such that
    \begin{equation}
        \label{ratio}
        \lim_{\n \in \Lambda'} \frac{Q_{\n^l,k}(z)}{Q_{\n,k}(z)} = G_{k} (z) , \qquad k=1,\ldots,m,
    \end{equation}
    uniformly on each compact subset of $\C \setminus \Delta_k$, where $G_{k} \in \mathcal{H}(\C \setminus \Delta_k)$. In principle, the limiting functions $G_k$ depend on $\Lambda'$. In order to prove the existence of limit along all $\Lambda$, it is sufficient to show that  $G_k= \widetilde{F}_k^{(l)}$ regardless of $\Lambda'$. Our goal will be accomplished with the aid of Lemma \ref{lm:BVP}.

    First, it is obvious that the functions $G_k$ and their reciprocals are analytic in $\C\setminus\Delta_k$. Therefore, condition 1 of Lemma \ref{lm:BVP} is fulfilled.
    On the other hand, considering the degrees of the polynomials $Q_{\n,k}$ and $Q_{\n^l,k}$, for all $\n \in \Lambda$ the rational functions on the left of \eqref{left} at infinity are either equal to 1 when $k=1, \ldots, l-1,$ or their derivative equals 1 for $k=l, \ldots, m$; hence, the limit functions must satisfy either $2a )$ or $2b)$  depending on $k$. Thus, any normalization of these functions obtained by means of a multiplication by positive constants also satisfies $1)$, $2a )$ or $2b)$.

    Now, to prove the boundary conditions 3 it is necessary to use some tools developed for the study of ratio and relative asymptotic of polynomials orthogonal with respect to varying measures. The main sources are \cite{Bernardo_Lago}, \cite{Lago87} and \cite{Lago89}.

    Define the constants
    \begin{align}
        K_{\n,k-1}:=&\left(\int Q_{\n,k}^2(x) \frac{|\mathcal{H}_{\n,k}(x)||\D \sigma_{k}(x)|}{|Q_{\n,k-1}(x)Q_{\n,k+1}(x)|}\right)^{-1/2},\quad k=1,\ldots,m\label{K_nk},\\
        K_{\n,m}:=& 1,\nonumber\\
        \kappa_{\n,k}:=& \frac{K_{\n,k-1}}{K_{\n,k}},\qquad k=1,\ldots,m.\nonumber
    \end{align}
    Set
    \begin{equation}
        \label{orthonormal}
        q_{\n,k} := \kappa_{\n,k}Q_{\n,k}, \quad h_{\n,k} := K_{\n,k}^2\mathcal{H}_{\n,k},\quad k=1,\ldots,m, \quad h_{\n,0} := K_{\n,0}^2\mathcal{H}_{\n,0}.
    \end{equation}
    With this notation the expression \eqref{orth_Qnj} is equivalent to
    \begin{equation*}
        \int x^\nu Q_{\n,k}(x)\frac{|h_{\n,k}(x)||\D \sigma_{k}(x)|}{|Q_{\n,k-1}(x)Q_{\n,k+1}(x)|} =0,\qquad \nu=0,\ldots,\eta_{\vec{n},k}-1.
    \end{equation*}
    Recall that $\sigma_k$ has constant sign and notice that $Q_{\n,k},Q_{\n,k-1}$ and $\mathcal{H}_{\n,k}$ have constant sign on $\Delta_k$. Therefore,  $Q_{\n,k}$ is the $\eta_{\vec{n},k}$-th monic orthogonal polynomial with respect to the varying measure
    \begin{equation*}
        \D \rho_{\n,k}(x) :=\frac{|h_{\n,k}(x)||\D \sigma_{k}(x)|}{|Q_{\n,k-1}(x)Q_{\n,k+1}(x)|},
    \end{equation*}
    and $q_{\n,k}$ is the $\eta_{\vec{n},k}$-th orthonormal polynomial with respect to the same varying measure.

    With an analogous reasoning, we have that $Q_{\n^l,k}$ is the $\eta_{\vec{n}^l,k}$-th monic orthogonal polynomial with respect to the varying measure
    \begin{equation}
        \label{Varying_measure}
        \frac{|h_{\n^l,k}(x)||\D \sigma_{k}(x)|}{|Q_{\n^l,k-1}(x)Q_{\n^l,k+1}(x)|} = \frac{|h_{\n^l,k}(x)|}{|h_{\n,k}(x)|} \frac{|Q_{\n,k-1}(x)Q_{\n,k+1}(x)|}{|Q_{\n^l,k-1}(x)Q_{\n^l,k+1}(x)|}\D \rho_{\n,k}(x).
    \end{equation}
    Using \eqref{ratio}, we deduce
    \begin{equation}
        \label{limit_Qn_Qn}
        \lim_{\vec n\in\Lambda'} \frac{|Q_{\n,k-1}(x)Q_{\n,k+1}(x)|}{|Q_{\n^l,k-1}(x)Q_{\n^l,k+1}(x)|} = \frac{1}{|G_{k-1}(z)G_{k+1}(z)|},\qquad k=1,\ldots,m,
    \end{equation}
    where the convergence is uniform on $\Delta_k$. On the other hand, from \eqref{Hnj_int} it follows that
    \begin{equation}
        \label{h_nk}
        |h_{\n,k}(z)| = \left|\int \frac{q_{\n,k+1}^2(x)}{z-x}\frac{|h_{\n,k+1}(x)||\D\sigma_{k+1}(x)|}{|Q_{\n,k}(x)Q_{\n,k+2}(x)|}\right|,\qquad k=0,\ldots,m-1.
    \end{equation}
    Moreover, we have the following relation between the degrees of the polynomials $Q_{\n,k}$, $Q_{\n,k+2}$ and $q_{\n,k+1}$
    \begin{align*}
        \deg Q_{\n,k}Q_{\n,k+2} - 2\deg q_{\n,k+1} =& \eta_{\vec{n},k-1} + \eta_{\vec{n},k+1}-2\eta_{\vec{n},k}\\
         =& n_{k+1}-n_{k}\leq N,
    \end{align*}
    where $N$ is the constant given in the assumptions which is independent of $\vec{n}\in\Lambda $. Consequently, taking into account \cite[Theorem 9]{Bernardo_Lago}, we obtain
    \begin{equation}
        \label{lim_hnj}
        \lim_{\n\in \Lambda} |h_{\n,k}(z)| = \frac{1}{|\sqrt{(z-b_{k+1})(z-a_{k+1})}|}, \qquad k=0,\ldots,m-1,
    \end{equation}
    uniformly on  each compact subset of $\C \setminus \Delta_{k+1}$, where $\Delta_{k+1} = [a_{k+1},b_{k+1}]$ (in particular on $\Delta_k$ when $k=1,\ldots m-1$).

    The proof of \eqref{lim_hnj} is carried out by induction for decreasing values of $k$. Indeed, if $k=m-1$, since $h_{\n,m} \equiv (-1)^m$, \eqref{h_nk} reduces to
    \[
    |h_{\n,m-1}(z)| = \left|\int \frac{q_{\n,m}^2(x)}{z-x}\frac{ |\D\sigma_{m}(x)|}{|Q_{\n,m-1}(x)|}\right|,
    \]
    and using \cite[Theorem 9]{Bernardo_Lago}, we obtain
    \[ \lim_{\n \in \Lambda} |h_{\n,m-1}(z)| = \left|\frac{1}{\pi} \int_{a_m}^{b_m} \frac{1}{z-x} \frac{\D x}{\sqrt{(b_m - x)(x-a_m)}} \right| = \frac{1}{|\sqrt{(z-a_m)(z-b_m)}|}
    \]
    pointwise for $z \in \C \setminus \Delta_m$. However, it is easy to verify that the family of functions $(h_{\n,m-1})_{\n \in \Lambda}$ is uniformly bounded on compact subsets of $\C \setminus \Delta_m$ and uniform convergence on compact subsets of that region follows from pointwise convergence.  Now, let $1\leq k+1 \leq m$ and assume that \eqref{lim_hnj}
    holds for $k+1$. Then, using \eqref{h_nk} we can apply once more \cite[Theorem 9]{Bernardo_Lago} to obtain \eqref{lim_hnj} for $k$, pointwise on $\C \setminus \Delta_{k+1}$, and uniform convergence follows as before.

    Similar arguments give
    \begin{equation}
        \label{lim_hnj2}
        \lim_{\n\in \Lambda} |h_{\n^l,k}(z)| = \frac{1}{|\sqrt{(z-b_{k+1})(z-a_{k+1})}|}, \qquad k=0,\ldots,m-1,
    \end{equation}
    uniformly on compact subsets of $\C \setminus \Delta_{k+1}$.

    By construction $h_{\n,m}\equiv h_{\n^l,m}\equiv (-1)^m$. Therefore, using \eqref{lim_hnj} and \eqref{lim_hnj2} it follows that
    \begin{equation}
        \label{limitQuot_hnk}
        \lim_{\n\in \Lambda} \frac{|h_{\n^l,k}(x)|}{|h_{\n,k}(x)|} = 1, \qquad k=1,\ldots,m,
    \end{equation}
    uniformly on $\Delta_{k}$. Putting together \eqref{limitQuot_hnk} and \eqref{limit_Qn_Qn}, we have
    \begin{equation}
        \label{limit_varying}
        \lim_{\vec n\in\Lambda'}\frac{|h_{\n^l,k(x)}|}{|h_{\n,k}(x)|} \frac{|Q_{\n,k-1}(x)Q_{\n,k+1}(x)|}{|Q_{\n^l,k-1}(x)Q_{\n^l,k+1}(x)|} =\frac{1}{|G_{k-1}(x)G_{k+1}(x)|}, \qquad k=1,\ldots,m,
    \end{equation}
    uniformly on the interval $\Delta_k$. The function on the right hand side of the previous expression is different from zero on $\Delta_k$.

    Fix $k=1,\ldots,m$. We distinguish two cases. If $k=1,\ldots,l-1\, (l \geq 2)$, then  $\deg Q_{\n^l,k} = \deg Q_{\n,k} = \eta_{\n,k}$. Using \eqref{Varying_measure} and \eqref{limit_varying}, the result on relative asymptotic of orthogonal polynomials with respect to varying measures which appears in \cite[Theorem 2]{DBG} implies that
    \begin{equation}
        \label{ratio_S}
        \lim_{\vec n\in\Lambda'} \frac{Q_{\n^l,k}(z)}{Q_{\n,k}(z)} = G_k(z) = \frac{\mathsf{S}_k(z)}{\mathsf{S}_k(\infty)}, \qquad k=1,\ldots,l-1,
    \end{equation}
    where $\mathsf{S}_k$ is the Szeg\H{o} function on $\overline{\C}\setminus\Delta_k$ with respect to the weight function
    \begin{equation*}
        |G_{k-1}(z)G_{k+1}(z)|^{-1},\qquad x\in\Delta_k.
    \end{equation*}
    Consequently,
    \begin{equation}
        \label{Sk_Gk}
        |\mathsf{S}_k(x)|^2|G_{k-1}(z)G_{k+1}(z)|^{-1}=1,\qquad x\in\Delta_k.
    \end{equation}
    and for $x\in \Delta_k$
    \begin{equation}\label{boundary1}
    \frac{|G_k(x)|^2}{|G_{k-1}(x)G_{k+1}(x)|} = \frac{1}{\mathsf{S}^2_k(\infty)}, \qquad k=1,\ldots,l-1.
    \end{equation}

    Now, if $k=l,\ldots,m,$ then $\deg Q_{\n^l,k} = \deg Q_{\n,k} + 1= \eta_{\n,k}+1$.
    Let $Q_{\n,k}^*$ be the $\eta_{\vec{n},k}$-th monic orthogonal polynomial with respect to the varying measure \eqref{Varying_measure}. Take
    \begin{equation*}
        \frac{Q_{\n^l,k}}{Q_{\n,k}} = \frac{Q_{\n^l,k}}{Q_{\n,k}^*}\frac{Q_{\n,k}^*}{Q_{\n,k}}.
    \end{equation*}
    For the second factor, reasoning as above, we get
    \begin{equation}
        \label{ratio_S^*}
        \lim_{\vec n\in\Lambda'} \frac{Q^*_{\n,k}(z)}{Q_{\n,k}(z)} = \frac{\mathsf{S}_k(z)}{\mathsf{S}_k(\infty)},
    \end{equation}
    where $\mathsf{S}_k$ is the same Szeg\H{o} function we had before.
    In the first factor,  we have the ratio of two monic  polynomials of consecutive degrees orthogonal with respect to the same varying measure  and with the help of the theorem on ratio asymptotic of orthogonal polynomials with respect to varying measures \cite[Theorem 6]{Bernardo_Lago} we deduce
    \begin{equation}
        \label{ratio_phi}
        \lim_{\n\in \Lambda'} \frac{Q_{\n^l,k}}{Q_{\n,k}^*}(z) = \frac{\varphi_k(z)}{\varphi_k'(\infty)},
    \end{equation}
    uniformly on compact subsets of $\C\setminus\Delta_j$, where $\varphi_k$ is the conformal representation of $\overline{\C}\setminus\Delta_k$ onto the exterior of the unit disc such that $\varphi_k(\infty)=\infty$ and $\varphi_k'(\infty)>0$.
    Combining \eqref{ratio} with \eqref{ratio_phi} and \eqref{ratio_S^*} we have
    \begin{equation}
        \label{Qnk_Gk_SK}
        \lim_{\vec n\in\Lambda'} \frac{Q_{\n^l,k}(z)}{Q_{\n,k}(z)} = G_{k}(z)= \frac{\mathsf{S}_k(z)\varphi_k(z)}{\mathsf{S}_k(\infty)\varphi_k'(\infty)},\qquad k=l,\ldots,m,
    \end{equation}
    and using \eqref{Sk_Gk} it follows that for $x\in\Delta_k$,
    \begin{equation}
        \label{boundary2}
        \frac{|G_{k}(x)|^2}{|G_{k-1}(x)G_{k+1}(x)|} =  \frac{1}{(\mathsf{S}_k(\infty)\varphi_k'(\infty))^2}, \quad k=l,\ldots,m.
    \end{equation}

    Putting together \eqref{boundary1} and \eqref{boundary2} we have proved that the collection of functions $(G_{k})_{k=1}^m$ satisfies the conditions of Lemma \ref{lm:BVP}, where the right hand side of 3) is $1/w_k$
    \begin{equation}
        \label{wk}
        w_k =
        \begin{array}{ll}
        (\mathsf{S}_k(\infty))^{2},& k = 1,\ldots,l-1, \\
        (\mathsf{S}_k(\infty)\varphi_k'(\infty))^{2},& k = l,\ldots,m,
        \end{array}
    \end{equation}
    (instead of $1$).

    Let $\widetilde{G}_k = c_k G_k,$ where $c_k, k=1,\ldots,m,$ are constants chosen appropriately so that
    \begin{equation*}
        \frac{ c_k^2}{w_kc_{k-1}c_{k+1}} = 1, \qquad k=1,\ldots,m \qquad (c_0=c_{m+1}=1).
    \end{equation*}
    Such constants exist. Indeed, taking logarithm we obtain the linear system of equations (in $\ln c_k$)
    \begin{equation}
        \label{linear_sys}
        2\ln c_k-\ln c_{k-1} -\ln c_{k+1} =\ln w_k,\qquad k=1,\ldots,m,
    \end{equation}
    which has a solution because the determinant of the system is different from zero. It is easy to verify that the collection of functions $(\widetilde{G}_k)_{k=1}^m$ satisfies all the conditions of Lemma \ref{lm:BVP}. Since that system of boundary value problems has only one solution, it follows that
    \[ \widetilde{G}_k = c_k G_k =  {F}_k^{(l)}, \qquad k=1,\ldots,m.
    \]
    Now, $G_k(\infty) = 1$ when $k=1,\ldots,l-1\, (l\geq 2)$, and $G_k'(\infty) = 1$ when $k=l,\ldots,m$; therefore, taking limit as $z\to \infty$ it follows that
    \begin{equation}
        \label{ck}
        c_k =  \left\{
        \begin{array}{ll}
            {F}_k^{(l)}(\infty), & k=1,\ldots,l-1,\\
            ({F}_k^{(l)})'(\infty), & k=l,\ldots,m.
        \end{array}
        \right.
    \end{equation}
    In any case, we have shown that independent of the subsequence $\Lambda'\subset \Lambda$ taken such that \eqref{ratio} takes place the limiting functions are
    \[ G_k = \widetilde{F}_k^{(l)}, \qquad k=1,\ldots,m,
    \]
    and \eqref{left} follows. With this we conclude the proof of Theorem \ref{th:ratio}.

    Since $a_{\n,m} = Q_{\n,m}$ for all $\n$, \eqref{left*} is a direct consequence of \eqref{left} and \eqref{boundary3} when $k=m$. Now, fix $k\in \{0,\ldots,m-1\}$ and $\varepsilon > 0$. Consider the positively oriented closed curve $\Gamma$ which surrounds $\Delta_m$   at distance $\varepsilon$. From \eqref{markov} and the argument principle it follows that
    \[ \lim_{\n\in \Lambda} \frac{1}{2\pi i}\int_{\Gamma} \frac{(a_{\n,k}/a_{\n,m})'(\zeta)}{(a_{\n,k}/a_{\n,m})(\zeta)} \D \zeta =  \frac{1}{2\pi i}\int_{\Gamma} \frac{\widehat{s}_{m,k+1}'(\zeta)}{\widehat{s}_{m,k+1}(\zeta)} \D \zeta = 1
    \]
    because $\widehat{s}_{m,k+1}$ has a simple zero at $\infty$ and no other zero or pole in all $\C \setminus \Delta_m$. The integrals on the left hand side only take integer values so they must be constantly equal to $1$ for all $\n \in \Lambda$ such that $|\n|$ is sufficiently large. Now, $\deg a_{\vec n,m} = |\n|$ and its zeros lie on $\Delta_m$ and $\deg a_{\n,k} \leq |\n|-1, k=0,\ldots,m-1$. It readily follows that for all $\n \in \Lambda$ with $|\n|$ sufficiently large, $\deg a_{\n,k} = |\n| -1$ and $a_{\n,k}$ has no zeros in the unbounded connected component of $\C \setminus \Gamma$. Since $\varepsilon > 0$ is arbitrary, we also obtain that the zeros of $a_{\n,k}$ accumulate on $\Delta_m$.

    Now, using \eqref{markov} and \eqref{left*} (for $k=m$) it follows that
    \[ \lim_{\n\in \Lambda} \frac{a_{\n^l,k}(z)}{a_{\n,k}(z)} = \lim_{\n\in \Lambda} \frac{a_{\n^l,k}(z)}{a_{\n^l,m}(z)} \frac{a_{\n,m}(z)}{a_{\n,k}(z)}\frac{a_{\n^l,m}(z)}{a_{\n,m}(z)} = \frac{{\psi}^{(l)}_m(z)}{({\psi}^{(l)}_m)'(\infty)},
	 \]
	uniformly on each compact subset of $\C\setminus\Delta_m$ and \eqref{left*} follows for $k=0,\ldots,m-1$.  \hfill $\Box$

    The next result complements Theorem \ref{th:ratio}.

    \begin{corollary}
        Assume that the conditions of Theorem \ref{th:ratio} hold.  Let $(q_{\n,k}=\kappa_{\n,k}Q_{\n,k})_{k=1}^m$, $\n\in \Lambda,$ be the system of orthonormal polynomials defined in \eqref{orthonormal} and $(K_{\n,k})_{k=1}^m$, $\n\in\Lambda,$ the values given in \eqref{K_nk}. Then, for each fixed $k=1,\ldots,m$ we have
        \begin{align}
            \lim_{\vec n\in\Lambda}\frac{\kappa_{\n^l,k}}{\kappa_{\n,k}} =& \kappa_k\label{kappa},\\
            \lim_{\vec n\in\Lambda}\frac{K_{\n^l,k-1}}{K_{\n,k-1}} =& \kappa_k\cdots\kappa_m,\label{Knk_asym}
        \end{align}
        and
        \begin{equation}
            \label{q_kappa_F}
            \lim_{\vec n\in\Lambda} \frac{q_{\n^l,k}(z)}{q_{\n,k}(z)} = \kappa_k {\widetilde{F}^{(l)}_k(z)} ,
        \end{equation}
        uniformly on compact subsets of $\C\setminus\Delta_k$, where
        \begin{equation}
            \label{def_ck}
            \kappa_k = \frac{c_k}{\sqrt{c_{k-1}c_{k+1}}},\qquad
            c_k =  \left\{
            \begin{array}{ll}
                {F}_k^{(l)}(\infty), & k=1,\ldots,l-1,\\
                ({F}_k^{(l)})'(\infty), & k=l,\ldots,m,
            \end{array}
            \right.
        \end{equation}
        and $c_0=c_m=1$. We also have
        \begin{equation}
            \label{ratio_Anj}
            \lim_n \left|\frac{\mathcal{A}_{\n^l,k}(z)}{\mathcal{A}_{\vec{n},k}(z)}\right| = \frac{1}{\kappa_{k+1}^2\cdots \kappa_{m}^2}
            \left|\frac{\widetilde{F}^{(l)}_k(z)}{\widetilde{F}^{(l)}_{k+1}(z)}\right|, \qquad k=0,\ldots,m-1,
        \end{equation}
        uniformly on compact subsets of $\C\setminus(\Delta_k\cup\Delta_{k+1})$. When $k=0$, $\Delta_0=\emptyset$.
    \end{corollary}
    \begin{proof}
        From  \eqref{left} it follows that in place of \eqref{limit_varying} we can write
        \begin{equation*}
            \lim_{\vec n\in\Lambda}\frac{|h_{\n^l,k}(x)|}{|h_{\n,k}(x)|} \frac{|Q_{\n,k-1}(x)Q_{\n,k+1}(x)|}{|Q_{\n^l,k-1}(x)Q_{\n^l,k+1}(x)|} =\frac{1}{| \widetilde{F}^{(l)}_{k-1}(x) \widetilde{F}^{(l)}_{k+1}(x)|}, \qquad k=1,\ldots,m.
        \end{equation*}
        By the same token, \eqref{ratio_S} and \eqref{Qnk_Gk_SK} hold with the limit taken along all $\Lambda$.

        With the same arguments that led to \eqref{ratio_S} and \eqref{Qnk_Gk_SK}, but in connection with orthonormal polynomials (see \cite{DBG} and \cite{Bernardo_Lago}) it follows that
        \begin{equation}
            \label{ortonormal}
            \lim_{\n\in \Lambda} \frac{q_{\n^l,k}(z)}{q_{\n,k}(z)} =
            \left\{
            \begin{array}{ll}
                {\mathsf{S}_k(z) }, & k=1,\ldots,l-1, \\
                 {\mathsf{S}_k(z)\varphi_k(z)}, & k=l,\ldots,m.
            \end{array}
            \right.
        \end{equation}
        uniformly on compact subsets of $\C\setminus\Delta_j$. Now, dividing \eqref{ortonormal} by \eqref{ratio_S} or \eqref{Qnk_Gk_SK}, we obtain
        \begin{equation*}
            \lim_n \frac{\kappa_{\n^l,k}}{\kappa_{\n,k}} = \sqrt{w_k} = \frac{c_k}{\sqrt{c_{k-1}c_{k+1}}} := \kappa_k,
        \end{equation*}
        where $w_k$ is given by \eqref{wk} and the $c_k$ are the normalizing constants found solving the linear system of equations \eqref{linear_sys} whose values were given in \eqref{ck}.  Therefore, formulas \eqref{kappa} and \eqref{def_ck} take place. Now, \eqref{Knk_asym} follows from \eqref{kappa} because
        \[ \frac{K_{\n^l,k-1}}{K_{\n,k-1}} = \frac{\kappa_{\n^l,k}\cdots\kappa_{\n^l,m}}{\kappa_{\n,k}\cdots\kappa_{\n,m}}
        \]
        and \eqref{q_kappa_F} from \eqref{kappa} and \eqref{left} since
        \[ \frac{q_{\n^l,k}}{q_{\n,k}} = \frac{\kappa_{\n^l,k}Q_{\n^l,k}}{\kappa_{\n,k}Q_{\n,k}}.
        \]

        From \eqref{Hnj}, \eqref{Hnj_int}, and  \eqref{orthonormal}, we deduce
        \begin{equation*}
            \mathcal{A}_{\vec{n},k}(z) = \frac{1}{K^2_{\n,k}}\frac{Q_{\n,k}(z)}{Q_{\n,k+1}(z)}\int \frac{q^2_{\n,k+1}}{z-x}\frac{h_{\n,k+1}(x)\D\sigma_{k+1}(x)}{Q_{\n,k}(x)Q_{\n,k+2}(x)},\qquad k=0,\ldots,m-1,
        \end{equation*}
        and similarly
        \begin{equation*}
            \mathcal{A}_{\vec{n}^l,k}(z) = \frac{1}{K^2_{\n^l,k}}\frac{Q_{\n^l,k}(z)}{Q_{\n^l,k+1}(z)}\int \frac{q^2_{\n^l,k+1}}{z-x}\frac{h_{\n^l,k+1}(x)\D\sigma_{k+1}(x)}{Q_{\n^l,k}(x)Q_{\n^l,k+2}(x)},\qquad k=0,\ldots,m-1,
        \end{equation*}
        Dividing the second expression by the first, taking absolute values,  and the limit over $\n \in \Lambda$ from  \eqref{left}, \eqref{Knk_asym}, \eqref{h_nk},   and  \eqref{limitQuot_hnk}, formula \eqref{ratio_Anj} readily follows.
    \end{proof}

\end{document}